\documentclass[12pt,reqno]{amsart}

\usepackage{amsmath,amsfonts,amsbsy,amsgen,amscd,mathrsfs,amssymb,amsthm,comment}
\usepackage{enumerate,mathtools}

\usepackage{amsmath, amsthm, amssymb, mathtools, graphicx, hyperref, color, braket, wasysym, geometry, bbm, color, xfrac, multirow}
\usepackage{tikz}
\usetikzlibrary{matrix,fit,decorations.pathreplacing,calc}

\usepackage{enumitem}
\setlist{parsep=0pt,listparindent=\parindent}
\usepackage{ wasysym }

\usepackage{tikz}
\usetikzlibrary{shadings,matrix,fit,decorations.pathreplacing,calc}

\newtheorem{thm}{Theorem}[section]
\newtheorem{lem}[thm]{Lemma}

\theoremstyle{definition}

\numberwithin{equation}{section} 
\numberwithin{figure}{section}
\numberwithin{table}{section}

\newcommand{\E}{\mathop{{}\mathbb{E}}}
\newcommand{\R}{\mathbb{R}}
\newcommand{\Z}{\mathbb{Z}}
\newcommand{\C}{\mathbb{C}}
\newcommand{\BH}{\mathrm{BH}}
\newcommand{\un}{\mathbf{I}}

\newcommand{\ketbra}[2]{\ket{#1}\!\!\bra{#2}}

\newcommand{\renorm}{\Gamma\!}

\newcommand{\bA}{\mathbf{A}}
\newcommand{\bB}{\mathbf{B}}
\newcommand{\bC}{\mathbf{C}}

\newcommand{\eps}{\varepsilon}

\newcommand{\al}{\alpha}

\newcommand{\cA}{\mathcal{A}}

\newcommand{\cS}{\mathcal{S}}

\newcommand{\cW}{\mathcal{W}}

\newcommand{\Om}{\Omega}
\newcommand{\om}{\omega}

\newcommand{\go}{K}

\newcommand{\tr}{\textnormal{tr}}

\newtheorem{theorem}{Theorem}
\newtheorem{proposition}[theorem]{Proposition}

\newtheorem{lemma}[theorem]{Lemma}
\newtheorem{corollary}[theorem]{Corollary}

\theoremstyle{definition}
\newtheorem{definition}[theorem]{Definition}
\newtheorem{remark}[theorem]{Remark}

\newtheorem*{question*}{Question}

\DeclareMathSymbol{\shortminus}{\mathbin}{AMSa}{"39}

\begin{document}

\title[Qudit  Bohnenblust--Hille inequality ]
{Noncommutative Bohnenblust--Hille Inequality for qudit systems}

\author{Joseph Slote}
\address{(J.S.) Department of Computing \& Mathematical Sciences,  California Institute of Technology, Pasadena, CA 91125}
\email{jslote@caltech.edu}

\author{Alexander Volberg}
\address{(A.V.) Department of Mathematics, MSU, 
East Lansing, MI 48823, USA and Hausdorff Center of Mathematics}
\email{volberg@math.msu.edu}

\author{Haonan Zhang}
\address{(H.Z.) Department of Mathematics, University of South Carolina, Columbia, SC, 29201}
\email{haonanzhangmath@gmail.com}

\begin{abstract}
Previous noncommutative Bohnenblust--Hille (BH) inequalities addressed operator decompositions in the tensor-product space $M_2(\C)^{\otimes n}$; \emph{i.e.,} for systems of qubits \cite{HCP22,VZ23}.
Here we prove noncommutative BH inequalities for operators decomposed in tensor-product spaces of arbitrary local dimension, \emph{i.e.,} $M_K(\C)^{\otimes n}$ for any $K\geq2$ or on systems of $K$-level qudits.
We treat operator decompositions in both the Gell-Mann and Heisenberg--Weyl basis, reducing to the recently-proved commutative hypercube BH \cite{DMP} and cyclic group BH \cite{SVZ} inequalities respectively.
As an application we discuss learning qudit quantum observables.
\end{abstract}



\thanks{
J.S. is supported by Chris Umans' Simons Foundation Investigator Grant.
The research of A.V. is supported by NSF DMS-1900286, DMS-2154402 and by Hausdorff Center for Mathematics.
H.Z. was supported by the Lise Meitner fellowship, Austrian Science Fund (FWF) M3337.
This work is partially supported by NSF  DMS-1929284 while all three authors were in residence at the Institute for Computational and Experimental Research in Mathematics in Providence, RI, during the Harmonic Analysis and Convexity program. Part of the work was finished when all the authors were visiting Stanford University, and we would like to thank their kind hospitality. }

\maketitle

\section{Quantum Bohnenblust--Hille inequalities}
\label{sec:BHineq}
Let 
\[
f(z) = \sum_\al c_\al z^\al=\sum_\al c_\al z^{\al_1}_1 \cdots z_n^{\al_n},
\]
where $\al=(\al_1, \dots, \al_n)$ are vectors of non-negative integers, all $c_\alpha$ are nonzero, and the total degree of polynomial $f$ is $d=\max_\al (\al_1+\dots+\al_n)$.
Here  $z$ can be all complex vectors in $\mathbb{T}^n=\{\zeta\in \C:|\zeta|=1\}^n$ or all sequences of $\pm1$ in Boolean cube $\{-1,1\}^n$. Bohnenblust--Hille-type of inequalities are the following
\begin{equation}
	\label{BHcom}
	\Big(\sum_\al |c_\al|^{\frac{2d}{d+1}}\Big)^{\frac{d+1}{2d}} \le C(d) \sup_{z}|f(z)|\,.
\end{equation}
The supremum is taken either over the torus $\mathbb{T}^n$ or, more recently, the Boolean cube $\{-1,1\}^n$.
In both cases this inequality is proven with constant $C(d)$ that is independent of the dimension $n$ and sub-exponential in the degree $d$.
More precisely, denote by $\textnormal{BH}^{\le d}_{\mathbb{T}}$ and $\textnormal{BH}^{\le d}_{\{\pm 1\}}$ the best constants in the Bohnenblust--Hille (BH) inequalities \eqref{BHcom} for polynomials of degree at most $d$ on $\mathbb{T}^n$ and $\{-1,1\}^n$, respectively. Then both $\textnormal{BH}^{\le d}_{\mathbb{T}}$ and $ \textnormal{BH}^{\le d}_{\{\pm 1\}}$ are bounded from above by $e^{c\sqrt{d\log d}}$ for some universal $c>0$ \cite{BPS,DMP}. The optimal dependence of $\textnormal{BH}^{\le d}_{\mathbb{T}}$ and $ \textnormal{BH}^{\le d}_{\{\pm 1\}}$ on the degree $d$ remains open.

Historically, the sub-exponential growth phenomenon of the optimal constant $\textnormal{BH}^{\le d}_{\mathbb{T}}$ in the Bohnenblust--Hille inequality \eqref{BHcom} for $\mathbb{T}^n$ has played an important role in resolving some open problems in functional analysis and harmonic analysis \cite{DGMS,BPS,DFOOS}. 

More recently, the Bohnenblust--Hille inequality \eqref{BHcom} on the Boolean cubes has found unexpected application in learning low-degree Boolean functions, a fundamental task in theoretical computer science \cite{Odonnell}.
For a long time, the best polytime algorithm for low-degree learning had a sample complexity exponentially separated in $n$ from what was information theoretically-required (poly($n$) vs. $\Omega(\log n)$) \cite{LMN,fouriergrowth}.
But in 2022 Eskenazis and Ivanisvili closed the gap \cite{EI22}, achieving a sample complexity of $\mathcal{O}(\log n)$ (optimal in $n$ due to \cite{EIS}) for the first time. The key ingredient in their argument is nothing but the Bohnenblust--Hille inequality \eqref{BHcom} on the Boolean cubes. We state it in full detail for later use. Recall that any function $f:\{-1,1\}^n\to \C$ admits the Fourier expansion
\begin{equation*}
f(x)=\sum_{A\subset [n]}\widehat{f}(A)\chi_A(x),\qquad x=(x_1,\dots, x_n)\in \{-1,1\}^n
\end{equation*}
where for each $A\subset [n]:=\{1,2,\dots, n\}$, $\chi_A$ is defined as $\chi_A(x)=\prod_{j\in A}x_j$. Then $f$ is of degree at most $d$ if $\widehat{f}(A)=0$ for $|A|>d$.

\begin{theorem}[Boolean Cube Bohnenblust--Hille \cite{Blei,DMP}]
\label{thm:boolean bh}
Fix $d\ge 1$. There exists a constant $C_d>0$ such that for all $n\ge 1$ and for all $f:\{-1,1\}^n\to [-1,1]$ of degree at most $d$, we have
\begin{equation*}
\|\widehat{f}\|_{\frac{2d}{d+1}}:=\left(\sum_{|A|\le d}|\widehat{f}(A)|^{\frac{2d}{d+1}}\right)^{\frac{d+1}{2d}}\le C_d \|f\|_{\infty}.
\end{equation*}
Moreover, the best constant satisfies $ \textnormal{BH}^{\le d}_{\{\pm 1\}}\le C^{\sqrt{d\log d}}$ for some universal $C>0$.
\end{theorem}

\medskip

One might ask if a similar approach to low-degree learning could work in the quantum world.
Quantum observables (Hermitian operators) on a system of $n$ qubits admit a ``Fourier-like'' decomposition 
\[
A = \sum_{\alpha\in\{0,1,2,3\}^n}\widehat{A}(\alpha)\sigma_\alpha \quad \text{where} \quad \sigma_\alpha = {\textstyle\bigotimes_{i=1}^n} \sigma_{\alpha_i}
\quad\text{and}\quad \widehat{A}(\alpha)=2^{-n}\mathop{\mathrm{tr}}[A\sigma_\alpha]\,.\]
Here $\sigma_0$ is the  2-by-2 identity matrix and $\sigma_i, 1\leq i \leq 3$ are the Pauli matrices
\[
\sigma_1=\begin{bmatrix}
0 & 1\\1 & 0
\end{bmatrix},\quad
\sigma_2=\begin{bmatrix}
0 & -i\\i & 0
\end{bmatrix},\quad
\sigma_3=\begin{bmatrix}
1 & 0\\0 & -1
\end{bmatrix}.\]
Defining $|\alpha|$ to be the number of nonzero entries in $\alpha$, we say $A$ is \emph{of degree at most $d$} if for all $\alpha$ with $|\alpha|>d$ we have $\widehat{A}(\alpha)=0$.
It was recently identified \cite{RWZ24,HCP22, VZ23} that this analogy is close-enough to the Boolean case that the Eskenazis--Ivanisvili approach goes through for quantum observables as well, provided a Bohnenblust--Hille-type inequality for Pauli decompositions exists, which was also proved:

\begin{theorem}[Qubit Bohnenblust--Hille \cite{HCP22, VZ23}]
    \label{thm:qubit-BH}
    Suppose that $d\ge 1$. Then there exists a constant $C_d>0$ such that for all $n\ge 1$ and all $A$ on $n$ qubits of degree at most $d$, we have
    \[\|\widehat{A}\|_{\frac{2d}{d+1}}:=\Big(\sum_{\alpha\in \{0,1,2,3\}^n}|\widehat{A}(\alpha)|^{\frac{2d}{d+1}}\Big)^{\frac{d+1}{2d}}\leq C_d\|A\|_{\textnormal{op}}\,,\]
    where $\|A\|_{\textnormal{op}}$ denotes the operator norm.
\end{theorem}

The qubit BH inequality, Theorem \ref{thm:qubit-BH}, has received two very different proofs.
In \cite{HCP22} Huang, Chen and Preskill pursue a direct proof and notably develop a physically-motivated ``algorithmic'' procedure to prove the key step in BH-type arguments known as \emph{polarization}.
They achieve the dimension-free constant $C_d\leq \mathcal{O}(d^d)$.
Another proof approach appears in \cite{VZ23}, which works by reducing the qubit BH inequality to the Boolean cube BH inequality.
Let $\BH^{\leq d}_{M_2}$ denote the optimal constant in Theorem \ref{thm:qubit-BH} (where $M_2$ designates the 2-by-2 complex matrix algebra).
Then \cite{VZ23} showed $\BH^{\leq d}_{M_2}\leq 3^d\BH^{\leq d}_{\{\pm 1\}}\leq C^{d}$.

Pauli matrices are very special objects, being Hermitian, unitary, and anticommuting, and it was unclear whether the reduction approach in \cite{VZ23} could be extended to the qudit setting, where higher-dimensional generalizations of Pauli matrices are not so well-behaved.
In fact we succeed in extending the reduction argument to different two bases for the complex matrix algebra $M_\go(\C)$, tensors of which form the appropriate space for qudit systems.
The bases under study are known as the (generalized) \emph{Gell-Mann basis} (GM basis for short) and the \emph{Heisenberg--Weyl basis} (HW basis for short), are orthonormal with respect to the normalized trace inner product $\frac{1}{\go}\tr[A^\dagger B]$, and are respectively Hermitian and unitary generalizations of the 2-dimensional Pauli basis.
Our proofs of these extensions reveal some pleasing features of the geometry of the eigenvectors of GM and HW matrices. 

To distinguish the Fourier expansions in these two bases, we use $\mathcal{A}$ to denote the observable in the GM basis, while we reserve $A$ for the HW basis. In the following, for fixed $\go\ge 2$ we use $e_j=\ket{j},1\le j\le \go$ to denote the canonical basis of $\C^\go$.

\subsection{Main results for the Gell-Mann Basis}

\begin{definition}[Gell-Mann Basis]
    Let $\go\ge 2$. Put $E_{jk}=\ket{j}\!\!\bra{k}, 1\le j,k\le \go$.
    The generalized Gell-Mann Matrices are a basis of $M_\go(\C)$ and are comprised of the identity matrix $\un$ along with the following generalizations of the Pauli matrices:
\begin{align*}
	\text{symmetric: }&& \bA_{jk} &= {\textstyle\sqrt{\!\frac{\go}{2}}}\big(E_{jk}+E_{kj}\big) & &\text{ for } 1\leq j<k\leq \go\\[0.3em]
	\text{antisymmetric: }&& \bB_{jk} &= {\textstyle\sqrt{\!\frac{\go}{2}}}\big(-iE_{jk}+iE_{kj}\big) & &\text{ for } 1\leq j < k\leq \go\\
	\text{diagonal: }&& \bC_m\, &= \renorm_{m}\left(\textstyle\sum_{k=1}^mE_{kk}-mE_{m+1,m+1}\right)  & &\text{ for } 1\leq m \leq \go-1,
\end{align*}
where $\renorm_{m} :=\sqrt{\!\frac{\go}{m^2+m}}$.
We denote
\[
\textnormal{GM(\go)}:= \{\un, \bA_{jk}, \bB_{jk}, \bC_m\}_{1\leq j<k\leq \go, 1\leq m \leq \go-1}\,.
\]
\end{definition}

An observable $\cA\in M_{\go}(\C)^{\otimes n}$ has expansion in the GM basis as
\[\cA = \sum_{\alpha\in \Lambda_\go^n}\widehat{\cA}(\alpha)M_\alpha=\sum_{\alpha\in \Lambda_\go^n}\widehat{\cA}(\alpha){\textstyle\bigotimes_{j=1}^n} M_{\alpha_j}\]
for some index set $\Lambda_\go$ (so $\{M_\alpha\}_{\alpha\in\Lambda_\go}=\mathrm{GM}(\go)$).
Letting $|\alpha|=|\{j:M_{\alpha_j}\neq \mathbf{I}\}|$, we say $\cA$ is \emph{of degree at most $d$} if $\widehat{\cA}(\alpha)=0$ for all $\alpha$ with $|\alpha|>d$.

We find the Gell-Mann BH inequality enjoys a reduction to the Boolean cube BH inequality on $\{-1,1\}^{n(\go^2-1)}$ and obtain the following.

\begin{theorem}[Qudit Bohnenblust--Hille, Gell-Mann Basis]\label{thm:GM}
	Fix any $\go\ge 2$ and $d\ge 1$. There exists $C(d,\go)>0$ such that for all $n\ge 1$ and GM observable $\cA\in M_\go(\C)^{\otimes n}$ of degree at most $d$, we have
	\begin{equation}\label{ineq:bh GM}
		\|\widehat{\cA}\|_{\frac{2d}{d+1}}\le C(d,\go)\|\cA\|_{\textnormal{op}}.
	\end{equation}
Moreover, we have $C(d,\go)\le \big(\tfrac32(\go^2-\go)\big)^d \textnormal{BH}^{\le d}_{\{\pm 1\}}$.
\end{theorem}
\noindent In particular, for $\go=2$ we recover the main result of \cite{VZ23} exactly. The proof of Theorem \ref{thm:GM} is contained in Section \ref{sec:bh GM}.

\subsection{Main results for the Heisenberg--Weyl Basis}

\begin{definition}[Heisenberg--Weyl Basis]
Fix $\go\geq 2$ and let $\omega =\omega_\go= \exp(2\pi i/\go)$.
Define the $\go$-dimensional \emph{clock} and \emph{shift} matrices respectively via
\begin{equation*}
Z\ket{j} = \om^j \ket{j},\qquad 
	X \ket{j}= \ket{j+1}\qquad\textnormal{for all} \qquad j\in \mathbb{Z}_\go.
\end{equation*}
Here $\Z_{\go}:=\{0,1,\ldots, \go-1\}$ denotes the additive cyclic group of order $\go$. Note that $X^\go=Z^\go=\un$. See more in \cite{AEHK}. 
Then the Heisenberg--Weyl basis for $M_\go(\C)$ is
\[
\textnormal{HW}(\go):=\{X^\ell Z^m\}_{\ell,m\in \Z_\go}\,.
\]
\end{definition}
\noindent Any observable $A\in M_\go(\bC)^{\otimes n}$ has a unique Fourier expansion with respect to $\textnormal{HW}(\go)$ as well:
\begin{equation}
    \label{expHW}
	A=\sum_{\vec{\ell},\vec{m}\in \mathbb{Z}_\go^n}\widehat{A}(\vec{\ell},\vec{m})X^{\ell_1}Z^{m_1}\otimes \cdots \otimes X^{\ell_n}Z^{m_n},
\end{equation}
where $\widehat{A}(\vec{\ell},\vec{m})\in\C$ is the Fourier coefficient at $(\vec{\ell},\vec{m})$. 
We say that $A$ is \emph{of degree at most $d$} if $\widehat{A}(\vec{\ell},\vec{m})=0$ whenever 
\begin{equation*}
	|(\vec{\ell},\vec{m})|:=\sum_{j=1}^{n}(\ell_j+m_j)>d.
\end{equation*}
Here, $0\le \ell_j,m_j\le \go-1$.

Noting that the eigenvalues of Heisenberg--Weyl matrices are the roots of unity, it is natural to pursue a reduction to a scalar BH inequality over $\Omega_\go^n$, the multiplicative cyclic group of order $\go$---precisely the inequality needed for classical learning on functions on $\Omega_\go^n$.
This reduction works well when $\go$ is prime.

\begin{theorem}[Qudit Bohnenblust--Hille, Heisenberg--Weyl Basis: prime case]\label{thm:bh HW prime}
	Fix a prime number $\go\ge 2$ and suppose $d\ge 1$.
    Consider an observable $A\in M_\go(\C)^{\otimes n}$ of degree at most $d$.
    Then we have
	\begin{equation}
        \label{ineq:bh-hw prime}
		\|\widehat{A}\|_{\frac{2d}{d+1}}\le C(d,\go)\|A\|_{\textnormal{op}},
	\end{equation}
	with $C(d,\go)\le (\go+1)^d \textnormal{BH}_{\Omega_\go}^{\le d}$.
\end{theorem}

When $\go$ is non-prime, the reduction still works under modifications, namely the degree may jump from $d$ up to $(\go-1)d$.

\begin{theorem}[Qudit Bohnenblust--Hille, Heisenberg--Weyl Basis: non-prime case]\label{thm:bh HW nonprime}
	Fix a non-prime number $\go\ge 4$ and suppose $d\ge 1$. Consider an observable $A\in M_\go(\C)^{\otimes n}$ of degree at  most $d$.
    Then we have
	\begin{equation}
        \label{ineq:bh-hw nonprime}
		\|\widehat{A}\|_{\frac{2(\go-1)d}{(\go-1)d+1}}\le C(d,\go)\|A\|_{\textnormal{op}},
	\end{equation}
	with $C(d,\go)\le \go^{2d} \textnormal{BH}_{\Omega_\go}^{\le (\go-1)d}$. In fact, the constant $\go^{2d}$ can be replaced by $|\Sigma_\go|^d$ with $|\Sigma_\go|$ being the cardinality of $\Sigma_\go=\{(\ell,m)\in \Z_\go\times \Z_\go: \ell \textnormal{ and } m \textnormal{ are coprime}\}$.
\end{theorem}

The proofs of Theorems \ref{thm:bh HW prime} and \ref{thm:bh HW nonprime} are contained in Section \ref{sec:hw-bh}. The full strength of Theorems \ref{thm:bh HW prime} and \ref{thm:bh HW nonprime} relies on the BH inequality for the cyclic groups $\Omega_\go^n$; \emph{i.e.}, the finiteness of the Bohnenblust--Hille constant $\textnormal{BH}_{\Omega_\go}^{\le d}$ for cyclic groups which we shall explain in more detail. This remained unclear for general $2<\go<\infty$ in the first version of this paper. Now it is fully proved \cite{SVZ,SVZgmp} and we below quote the results to clarify that the constants $C(d,\go)$ in \eqref{ineq:bh-hw prime} and \eqref{ineq:bh-hw nonprime} are indeed dimension-free.

\medskip 

Fix $\go\ge 3$ and denote $\omega:=e^{2\pi i/\go}.$ Let $\Omega_{\go}:=\{1,\omega,\omega^2,\dots, \omega^{\go-1}\}.$ Then any function $f:\Omega
_{\go}^n\to \C$ admits the unique Fourier expansion 
\begin{equation}
f(z)=\sum_{\alpha}\widehat{f}(\alpha)z^\alpha, 
\end{equation}
where $\alpha=(\alpha_1,\dots, \alpha_n)$ are vectors of non-negative integers and each $\alpha_j\le \go-1$. We say $f$ is \emph{of degree at most $d$} if $\widehat{f}(\alpha)=0$ whenever $|\alpha|>d.$ The following result was proved in \cite{SVZ,SVZgmp,ksvzITCS,BKSVZremez}.

\begin{theorem}[Cyclic Bohnenblust--Hille]
\label{thm:bh cyclic}
Fix $\go\ge 3$ and $d\ge 1$. There exists $C(d,\go)>0$ such that for all $n\ge 1$ and for all $f:\Omega
_{\go}^n\to \C$ of degree at most $d$, we have 
\begin{equation}\label{ineq:cyclic bh}
\|\widehat{f}\|_{\frac{2d}{d+1}}\le C(d,\go)\sup_{z\in \Omega_{\go}^n}|f(z)|.
\end{equation} 
\end{theorem}

Denote by $\textnormal{BH}_{\Omega_\go}^{\le d}$ the best constant $C(d,\go)$ in \eqref{ineq:cyclic bh}. An upper bound $\textnormal{BH}_{\Omega_\go}^{\le d}\le O(\log K)^d$ was obtained in \cite{ksvzITCS,BKSVZremez}.

\subsection{Applications}

We refer to \cite{VZ23} and references therein for several applications of (noncommutative) Bohnenblust--Hille inequalities in mathematics. Here we focus on applications to learning theory. See Section \ref{sec:learning} for details.



A central motivation for the development of the noncommutative \emph{qubit} BH inequality was to generalize the Eskenazis--Ivanisvili result on learning low-degree functions to quantum setting of learning low-degree observables.
And here it could be argued that the general \emph{qudit} case is even more important, both for the study of fundamental physics via quantum simulation (\emph{e.g.}, \cite{quant-sim,PhysRevLett.129.160501}) and in the operation and validation of quantum computers.
In both contexts, gains in efficiency are possible when the underlying hardware system is composed of higher-dimensional subsystems, sometimes carrying an algorithm from theoretical fact to practical reality in the NISQ era \cite{PhysRevLett.129.160501}---and this benefit may very well remain as quantum computing advances.
Such systems are called \emph{multilevel system}, or \emph{qudit} quantum computers \cite{quditsurvey}.
While the qu\emph{b}it case gives a conceptual sense of the possibilities for learning on qudit systems, it is practically important to derive guarantees and algorithms that work directly in the native dimension of the quantum system.
In so doing we also establish new distributions under which arbitrary quantum processes are well-approximated by low-degree ones (including new distributions for qubit systems beyond those identified in \cite{HCP22}).

\begin{theorem}[Qudit Observable Learning, Informal]
    \label{thm:qudit-learning}
    Let $\cA$ be any (not necessarily low-degree) bounded quantum observable on $\mathcal{H}_\go^{\otimes n}$; \emph{i.e.,} on $n$-many $\go$-level qudits.
    Then we may via random sampling construct in polynomial time an approximate observable $\widetilde{\cA}$ such that for a wide class of distributions $\mu$ on states $\rho$,
    \[\E_{\rho\sim\mu}\big|\tr[\cA\rho]-\tr[\widetilde{\cA}\rho]\big|^2\leq \eps\,.\]
    The samples are of the form $(\rho, \tr[\cA \rho])$ for $\rho$ drawn from the uniform distribution over a certain set of product states.
    Moreover, the number of samples $s$ required to achieve the guarantee with confidence $1-\delta$ is
    \[s= \mathcal{O}\Big(\log(\tfrac{n}{\delta})\,C^{\log^2(1/\epsilon)}(C'\go\|\cA^{\leq t}\|_\mathrm{op})^{2t}\Big)\,.\]
   Here $\cA^{\leq t}$ denotes the truncation of $\cA$ up to degree $t$, for $t$ roughly $\log(1/\epsilon)$.
\end{theorem}
\noindent There is little prior work on learning qudit observables, but earlier polytime algorithms for learning qubits required at least $\mathcal{O}(n\log(n))$ samples to complete a comparable task \cite{Huang2020}.
Theorem \ref{thm:qudit-learning} is proved by combining a low-degree learning algorithm for qudits with a qudit extension of the low-degree approximation lemma of Huang, Chen and Preskill \cite[Lemma 14]{HCP22}.
The distributions $\mu$ admitting this construction are studied in Section \ref{sec:learning-arb}.

\bigskip 

\noindent\textbf{Notation.} In what follows $\langle\cdot,\cdot\rangle$ denotes the inner product on $\C^n$ that is linear in the second argument. For any operator $A$, we denote by $A^\dagger$ the adjoint of $A$ with respect to $\langle\cdot,\cdot\rangle$.

\section{Qudit Bohnenblust--Hille in the Gell-Mann basis}
\label{sec:bh GM}

In this section we prove Theorem \ref{thm:GM} by reducing \eqref{ineq:bh GM} to the Boolean cube Bohnenblust--Hille inequality on $\{-1, 1\}^{n (\go^2-1)}$.
The $\go=2$ case was done in \cite{VZ23}.

The central part of the reduction is a coordinate-wise construction of  density matrices $\rho(\boldsymbol{x})\in M_\go(\C)$ parametrized by $\boldsymbol{x}\in \{-1,1\}^{\go^2-1}=:H_\go$.
It will be convenient to partition the coordinates of
$\boldsymbol{x}$ as 
$$\boldsymbol{x}=(x,y,z)\in \{-1, 1\}^{\binom{\go}{2}}\times \{-1, 1\}^{\binom{\go}{2}}\times \{-1, 1\}^{\go-1}$$
with indices
\begin{align*}
	x=(x_{jk})_{1\le j<k\le \go}, \qquad y=(y_{jk})_{1\le j<k\le \go},\qquad 
 \text{and} \qquad z=(z_m)_{1\le m\le \go-1}\,.
\end{align*}

\begin{lemma}
	\label{action}
	For any $(x,y,z)\in H_\go$, there exists a density matrix $\rho=\rho(x,y,z)$ such that for all $1\le j<k\le \go$ and $1\le m \le \go-1$,
	\begin{align}
    \label{action-xs}
	\tr [\bA_{jk} \rho(x, y, z)] &= {\textstyle \frac13\binom{\go}{2}^{-1}\sqrt{\!\frac{\go}{2}}}x_{jk},\\
    \label{action-ys}
	\tr [\bB_{jk} \rho(x, y, z)] &= {\textstyle \frac13\binom{\go}{2}^{-1}\sqrt{\!\frac{\go}{2}}}y_{jk},\\
    \label{action-zs}
	\tr [\bC_{m} \rho(x, y, z)] &= {\textstyle \frac13\binom{\go}{2}^{-1}\sqrt{\!\frac{\go}{2}}}z_{m}.
	\end{align}
\end{lemma}
\begin{proof}

For $b\in\{-1,1\}$ and $1\le j<k\le \go$ consider unit vectors
\[
\al_{jk}^{(b)} = (e_j+be_k)/\sqrt{2}, \qquad \beta_{jk}^{(b)}=(e_j+b i e_k)/\sqrt{2}\,.
\]
These are respectively eigenvectors of $\bA_{jk}$ and $\bB_{jk}$ corresponding to the eigenvalue $b\sqrt{\!\go/2\,}$. 
Now consider the density matrices, again for $b\in\{-1,1\}$ and $1\le j<k\le \go$,
\[
A_{jk}^{(b)}= \ketbra{\al_{jk}^{(b)}}{\al_{jk}^{(b)}} ,\qquad B_{jk}^{(b)}= \ketbra{\beta_{jk}^{(b)}}{\beta_{jk}^{(b)}}.
\]
Finally, define
\[
\rho=\rho(x, y, z) =\frac13\binom{\go}{2}^{-1}\left(\sum_{1\le j<k\le \go}  A_{jk}^{(x_{jk})} + \sum_{1\le j<k\le \go} B_{jk}^{(y_{jk})} + \sum_{m=1}^{\go-1}  \tfrac{z_m}{\sqrt{2\go}}\bC_m+ \tfrac{\go-1}{2}  \un\right)\,.
\]
Observe $\rho$ is a positive semi-definite Hermitian matrix: each $A_{jk}^{(x_{jk})},B_{jk}^{(y_{jk})}$ are positive semi-definite Hermitian and the remaining summands form a diagonal matrix with positive entries.
Also, we have 
\begin{equation}
	\label{trace-rho}
	\tr \,\rho = \frac13\binom{\go}{2}^{-1}\left(\frac{\go(\go-1)}{2}+\frac{\go(\go-1)}{2}+0+\frac{\go(\go-1)}{2}\right)=1\,.
\end{equation}

By definition, we have 
\begin{equation*}
\tr(\bA_{jk} A_{jk}^{(x_{jk})})=\sqrt{\frac{\go}{2}}x_{jk},\qquad \tr(\bB_{jk} A_{jk}^{(y_{jk})})=\sqrt{\frac{\go}{2}}y_{jk}.
\end{equation*}
	Note for any $1\leq j<k\leq \go$ the anti-commutative relationship
	$\bA_{jk} \bB_{jk} + \bB_{jk}\bA_{jk}=0.$
    This implies that (see for example \cite[Lemma 2.1]{VZ23} or Lemma \ref{lem:orthogonal} below) for any $b\in\{-1,1\}$,
    $
	\braket{ \beta_{jk}^{(b)},\bA_{jk}\beta_{jk}^{(b)}} = \braket{ \alpha_{jk}^{(b)},\bB_{jk}\alpha_{jk}^{(b)}}=0,
    $
    and thus
    \begin{equation*}
    \tr[\bA_{jk} B_{jk}^{(b)}]=\tr[\bB_{jk} A_{jk}^{(b)}]=0\,.
    \end{equation*}
	When $(j, k)\neq (j', k')$ then the operators ``miss'' each other and we get for all $b\in\{-1,1\}$ 
	\begin{equation*}
		\tr[\bA_{jk} B_{j'k'}^{(b)}]=\tr[\bB_{jk} A_{j'k'}^{(b)}]=\tr[\bA_{jk} A_{j'k'}^{(b)}]=\tr[\bB_{jk} B_{j'k'}^{(b)}]=0.
	\end{equation*}
    By orthogonality the remaining summands in $\rho$ contribute $0$ to $\tr(\bA_{jk}\rho),\tr(\bB_{jk}\rho)$.
    We conclude \eqref{action-xs} and \eqref{action-ys} hold.
    
	So far all follows more or less the path of \cite{VZ23}. 
	Now we claim
	\begin{equation}
		\label{eq:claim}
		\tr \bigg[\bC_m \Big(\sum_{1\le j<k\le \go}    A_{jk}^{(x_{jk})}+    B_{jk}^{(y_{jk})}\Big)\bigg]=0.
	\end{equation}
	 To see this note that $\bC_m$ is diagonal, so we only need to consider the diagonal part of the summation in \eqref{eq:claim}. By definition, the diagonal parts of $\sum_{1\le j<k\le \go}    A_{jk}^{(x_{jk})}$ and  $\sum_{1\le j<k\le \go}    B_{jk}^{(y_{jk})}$ are both 
    $$\sum_{1\le j<k\le \go}\frac{1}{2}(e_j+e_k)=\frac{\go-1}{2}\un$$
    regardless the choices of $(x_{jk})$ and $(y_{jk})$. 
    So
	\begin{align*}
\tr \bigg[\bC_m \Big(\sum_{1\le j<k\le \go}A_{jk}^{(x_{jk})}+B_{jk}^{(y_{jk})}\Big)\bigg]
&=(\go-1)\tr[\bC_m]=0,
\end{align*}
which finishes the proof of the claim \eqref{eq:claim}.

Now consider the rest of $\rho(x, y, z)$ which is $\sum_{m=1}^{\go-1}  \frac{z_m}{\sqrt{2\go}}\bC_m + \frac{\go-1}{2}\un$. By orthogonality, we have 
\begin{equation*}
\tr\left[\bC_m \left(\sum_{m=1}^{\go-1}  \frac{z_m}{\sqrt{2\go}}\bC_m + \frac{\go-1}{2}\un\right)\right]
=\frac{z_m}{\sqrt{2\go}}\tr[\bC_m^2]
=\frac{z_m}{\sqrt{2\go}}\Gamma_m^2 m(m+1)
=\sqrt{\frac{\go}{2}}z_m.
\end{equation*}
All combined, we complete the proof of \eqref{action-zs}.
	\end{proof}

Now we are ready to prove Theorem \ref{thm:GM}.

\begin{proof}[Proof of Theorem \ref{thm:GM}]
Choose any $(\vec x, \vec y, \vec z) \in H_\go^n$ with 
\begin{equation*}
	\vec x=\big(x^{(1)},\dots, x^{(n)}\big), \qquad
	\vec y=\big(y^{(1)},\dots, y^{(n)}\big), \qquad
	\vec z=\big(z^{(1)},\dots, z^{(n)}\big), 
\end{equation*}
and $\big(x^{(j)},y^{(j)},z^{(j)}\big)\in H_\go$ for each $1\le j\le n$.
We can consider a tensor of $r$ states
\[
\rho(\vec x, \vec y, \vec z) := \rho\big(x^{(1)}, y^{(1)}, z^{(1)}\big) \otimes  \rho\big(x^{(2)}, y^{(2)}, z^{(2)}\big) \otimes\dots \otimes \rho\big(x^{(n)}, y^{(n)}, z^{(n)}\big)\,.
\]
Recall that any GM observable $\cA$ of degree at most $d$ has the unique expansion
\begin{equation*}
	\cA=\sum_{\al=(\al_1,\dots, \al_n)\in \Lambda_\go^n}\widehat{\cA}_\al M_{\al_1}\otimes \cdots \otimes M_{\al_n}
\end{equation*}
where $\{M_\al \}_{\al\in  \Lambda_\go}=\textnormal{GM}(\go)$ and $\widehat{\cA}_\al=0$ if more than $d$ matrices of $M_{\al_j}, 1\le j\le n$ are not identity matrices.

By Lemma \ref{action}, for any $\al=(\al_1,\dots, \al_n)\in \Lambda_\go^n$ with $|\alpha|=:\kappa\le d$, the map
\begin{equation*}
	(\vec x,\vec y,\vec z)\mapsto \tr \left[{\textstyle\bigotimes_{j=1}^n M_{\al_1}}\cdot \rho(\vec x, \vec y, \vec z)\right]
\end{equation*}
is a multi-linear monomial of degree $\kappa\le d$ on the Boolean cube $H_\go^n = \{-1, 1\}^{n(\go^2-1)} $ with coefficient 
\[
\left(\frac{\sqrt{\go/2}}{3\binom{\go}{2}}\right)^\kappa\ge\big(\tfrac32(\go^2-\go)\big)^{-d} .
\]
Moreover, for distinct $\alpha$'s, the corresponding classical monomials are different. So for a GM observable $\cA$ of degree at most $d$, the function 
\begin{equation}
f_{\cA}(\vec x,\vec y,\vec z):=\tr\left[\cA\cdot \rho(\vec x, \vec y, \vec z)\right]
\end{equation}
defined on $ \{-1, 1\}^{n(\go^2-1)} $ is of degree at most $d$, and 
\begin{equation}
\|\widehat{\cA}\|_p\le \big(\tfrac32(\go^2-\go)\big)^d \|\widehat{f_\cA}\|_p, \qquad p> 0.
\end{equation}
By Theorem \ref{thm:boolean bh} one has
\begin{equation*}
 \|\widehat{f_\cA}\|_{\frac{2d}{d+1}}\le  \textnormal{BH}^{\le d}_{\{\pm 1\}}\sup_{(\vec x,\vec y,\vec z)\in H_\go^n}|\tr(\cA\cdot \rho(\vec x, \vec y, \vec z)|\, .
\end{equation*}
Since $\rho(\vec x, \vec y, \vec z)$ is a density matrix, we have
\[
\sup_{(\vec x,\vec y,\vec z)\in H_\go^n}|\tr[\cA\cdot \rho(\vec x, \vec y, \vec z)]|\le \|\cA\|_{\textnormal{op}}\,.
\]
All combined, we get 
\begin{equation*}
\|\widehat{\cA}\|_{\frac{2d}{d+1}} \le \big(\tfrac32(\go^2-\go)\big)^d \textnormal{BH}^{\le d}_{\{\pm 1\}} \|\cA\|_{\textnormal{op}}\,.
\qedhere
\end{equation*}
\end{proof}

\section{Qudit Bohnenblust--Hille in the Heisenberg--Weyl basis}
\label{sec:hw-bh}

In this section we prove Theorems \ref{thm:bh HW prime} and \ref{thm:bh HW nonprime} via reduction to the BH inequality for cyclic groups, Theorem \ref{thm:bh cyclic}.
We collect first a few facts about the Heisenberg--Weyl basis $\{X^\ell Z^m\}$.

Fix $\go \ge 3$. Recall that $\gcd(a,b)$ denotes the 
greatest common divisor of $a$ and $b$. For $(\ell,m)\in \Z_{\go}\times \Z_{\go}$, $\gcd(\ell,m)$ is understood as when $\ell,m\in \{1,2,\dots, \go\}$, i.e. we do not mod $\go$ freely here. For example if $\go=6$, then $\gcd(0,2)$ is understood as $\gcd(6,2)=2.$  
For a group $G$, we use the convention that $\langle g\rangle$ is the abelian subgroup generated by $g\in G$. So for any $(\ell,m)\in \Z_{\go}\times \Z_{\go}$, we have 
\begin{equation}
\langle (\ell,m)\rangle=\{(k\ell,km):k\in \Z_{\go}\}.
\end{equation} 

In the sequel, we denote $\omega=\omega_{\go}=e^{2\pi i/\go}$ and use the notation $\omega^{1/2}:=\omega_{2\go}=e^{\pi i/\go}$.

\begin{lemma}\label{lem:HW basis general}
	We have the following:
	\begin{enumerate}
		\item $\{X^\ell Z^m: \ell, m\in \mathbb{Z}_\go\}$ form a basis of $M_\go(\C)$.
\item For all $k,\ell,m\in \mathbb{Z}_\go$:
		\begin{equation*}
			(X^\ell Z^m)^k=\om^{\frac{1}{2}k(k-1)\ell m}X^{k\ell} Z^{km}
		\end{equation*}
		and for all $\ell_1,\ell_2,m_1,m_2\in\mathbb{Z}_\go$:
		\begin{equation*}
			X^{\ell_1}Z^{m_1}\cdot X^{\ell_2}Z^{m_2}=\om^{\ell_2 m_1-\ell_1 m_2} X^{\ell_2}Z^{m_2}\cdot X^{\ell_1}Z^{m_1}.
		\end{equation*}
\item If $\gcd(\ell_1,m_1)=1$ and $(\ell,m)\notin \langle(\ell_1,m_1)\rangle $, then
\begin{equation}
X^{\ell_1}Z^{m_1}\cdot X^{\ell}Z^{m}
= \omega^{\ell m_1-\ell_1 m}X^{\ell}Z^{m}\cdot X^{\ell_1}Z^{m_1}
\end{equation}
with $\omega^{\ell m_1-\ell_1 m}\neq 1$.

\item If $\gcd(\ell,m)=1$, then the set of eigenvalues of $X^\ell Z^m$ is either $\Omega_\go$ or $\Omega_{2\go}\setminus \Omega_{\go}$.
	\end{enumerate}
\end{lemma}

\begin{proof}
	\begin{enumerate}
		\item Suppose that $\sum_{\ell,m}a_{\ell,m}X^\ell Z^m=0$. For any $j,k\in \mathbb{Z}_\go$, we have
		\begin{equation*}
			\sum_{\ell,m} a_{\ell,m}	\langle X^\ell Z^m e_j,e_{j+k}\rangle 
			=\sum_{m} a_{k,m}	\om^{jm}
			=0.
		\end{equation*}
		Since the Vandermonde matrix associated to $(1,\om,\dots,\om^{\go-1})$ is invertible, we have $a_{k,m}=0$ for all $k,m\in \mathbb{Z}_\go$.

\item It follows immediately from the identity $ZX=\om XZ$ which can be verified directly: for all $j\in \mathbb{Z}_\go$
		\begin{equation*}
			ZXe_j= Ze_{j+1}=\om^{j+1}e_{j+1}=\om^{j+1} X e_{j}=\om XZ e_{j}.
		\end{equation*}
		\item It is a direct consequence of (2) and the following fact:
for $(\ell_1, m_1)\in\Z_{\go} \times \Z_{\go}$ such that $\gcd(\ell_1, m_1)=1$ and $(\ell_2,m_2)\in \Z_{\go} \times \Z_{\go}$, we have $\ell_1 m_2-\ell_2 m_1 \equiv 0\mod \go$ if and only if $(\ell_2, m_2)\equiv (k \ell_1, k m_1)\mod \go$ for some $k\in \Z_{\go}$.

The ``if" direction is obvious. To show the ``only if" part, recall that by B\'ezout's lemma, there exist integers $\alpha$ and $\beta$ such that $\alpha \ell_1+\beta m_1=\gcd(\ell_1,m_1)=1$. Take $k\equiv \alpha \ell_2+\beta m_2 \mod \go$. Then 
\begin{equation}
\ell_2=\ell_2 (\alpha \ell_1+\beta m_1)\equiv \alpha\ell_1\ell_2+\beta \ell_1 m_2
\equiv k\ell_1 \mod \go,
\end{equation}
where we used $\ell_1 m_2\equiv \ell_2 m_1 \mod \go$. Similarly, 
\begin{equation}
m_2=m_2 (\alpha \ell_1+\beta m_1)\equiv \alpha\ell_2 m_1+\beta m_1 m_2
\equiv km_1 \mod \go,
\end{equation}
as desired. This finishes the proof of the fact.
\item By (2), we have 
$$
(X^{\ell}Z^m)^{2\go}=\omega^{\go(2\go-1)\ell m}X^{2\ell \go}Z^{2m\go}=\un.
$$		
So the eigenvalues of $X^{\ell}Z^m$ must be roots of unit of order $2\go$. Then the proof is finished as soon as we prove the following claim: for $\gcd(\ell,m)=1$, if $\lambda$ is an eigenvalue of $X^\ell Z^m$, then so is $\omega \lambda$. To prove the claim, recall that by B\'ezout's lemma, there exist integers $\alpha$ and $\beta$ such that $\alpha\ell+\beta m=\gcd(\ell,m)=1$. By (2), we get
\begin{equation}
X^\ell Z^m X^{\beta} Z^{-\alpha}=\omega^{\alpha \ell +\beta m}X^{\beta} Z^{-\alpha} X^\ell Z^m=\omega X^{\beta} Z^{-\alpha} X^\ell Z^m.
\end{equation}
Suppose $\vec{0}\neq \xi$ is an eigenvector of $X^\ell Z^m$ with eigenvalue $\lambda$. Then 
\begin{equation}
X^\ell Z^m X^{\beta} Z^{-\alpha}\xi
=\omega^{\alpha \ell +\beta m}X^{\beta} Z^{-\alpha} X^\ell Z^m\xi
=\omega \lambda  X^{\beta}Z^{-\alpha} \xi,
\end{equation}
implying that $X^{\beta}Z^{-\alpha} \xi$ (non-zero since  $X^{\beta}Z^{-\alpha}$ is invertible) is an eigenvector of $X^\ell Z^m$ with eigenvalue $\omega \lambda$. This finishes the proof of the claim. \qedhere
\end{enumerate}
\end{proof}

Let us record the following observation as a lemma. 

\begin{lemma}\label{lem:orthogonal}
	Suppose that $k\ge 1$, $A,B$ are two unitary matrices such that $B^k=\un$, $AB=\lambda BA$ with $\lambda\in\C$ and $\lambda\ne 1$. 
	If $\xi\neq \vec{0}$ is an eigenvector of $B$ with eigenvalue $\mu$ ($\mu\ne 0$ since $\mu^k=1$), then 
	\begin{equation*}
		\langle \xi,A\xi\rangle=0.
	\end{equation*}
\end{lemma}

\begin{proof}
	By assumption
	\[
	\mu\langle \xi,A\xi\rangle
	=\langle \xi,AB\xi\rangle 
	=\lambda \langle \xi,BA\xi\rangle.
	\]
	Since $B^\dagger=B^{k-1}$, $B^\dagger\xi=B^{k-1}\xi=\mu^{k-1}\xi=\overline{\mu}\xi$. Thus
	\[
	\mu\langle \xi,A\xi\rangle
	=\lambda \langle \xi,BA\xi\rangle
	=\lambda \langle B^\dagger\xi,A\xi\rangle
	=\lambda \mu\langle \xi,A\xi\rangle.
	\]
	Hence, $\mu(\lambda-1)\langle \xi,A\xi\rangle=0$. This gives $\langle \xi,A\xi\rangle=0$ as $\mu(\lambda-1)\ne 0$. 
\end{proof}

\subsection{The prime $\go$ case}

In this subsection we prove Theorem \ref{thm:bh HW prime}. 
When $\go$ is prime, the basis $\{X^\ell Z^m\}$ has nicer properties. 

\begin{lemma}\label{lem:HW basis prime}
Fix $\go\ge 3$ a prime number. Consider the set of generators 
\begin{equation}\label{eq:Sigma prime}
\Sigma_{\go}:=\{(1,0),(1,1),\dots, (1,\go-1),(0,1)\}.
\end{equation}
Then the group $\Z_{\go}\times \Z_{\go}$ is the union of subgroups 
\begin{equation}
\Z_{\go}\times \Z_{\go}=\bigcup_{(\ell,m)\in \Sigma_{\go}}\langle (\ell,m)\rangle
\end{equation}
where each two subgroups intersects with the unit $(0,0)$ only. Moreover, for any $(\ell,m)\in \Sigma_\go$, the set of eigenvalues of each $X^\ell Z^m$ is exactly $\Omega_\go.$
\end{lemma}

\begin{proof}
To prove the first statement, take any $(\ell,m)\in \Z_\go\times \Z_\go$. If $\ell=0$, then $(\ell,m)=(0,m)\in \langle (0,1) \rangle$. If $\ell\neq 0$, then $\gcd(\ell,\go)=1$. So by B\'ezout's lemma, there exists $\ell'$ such that $\ell \ell'\equiv 1\mod \go$. Thus $(\ell,m)=(\ell, \ell\ell'm)\in \langle (1,\ell'm )\rangle$. The statement about the intersection is clear since otherwise, the cardinality of the union of these subgroups is strictly smaller  than $1+(\go+1)(\go-1)=\go^2$ which leads to a contradiction.

 The second statement follows from the proof of Lemma \ref{lem:HW basis general}. In fact, when $\go$ is odd, $(\go-1)/2$ is an integer and we have by Lemma \ref{lem:HW basis general} (2) that
$$
(X^{\ell}Z^m)^{\go}=\omega^{\frac{1}{2}\go(\go-1)\ell m}X^{\ell \go}Z^{m\go}=\un.
$$	
So the eigenvalues of $X^{\ell}Z^m$ must be roots of unity of order $\go$. This, together with the claim in the proof of Lemma \ref{lem:HW basis general} (4) and the fact that $\gcd(\ell,m)=1, (\ell,m)\in \Sigma_\go$, completes the proof of the lemma. For concrete forms of the eigenvectors, see the Appendix \ref{appendix} for details.
\end{proof}

Now we are ready to prove Theorem \ref{thm:bh HW prime}:
\begin{proof}[Proof of Theorem \ref{thm:bh HW prime}]
Fix a prime number $\go\ge 2$. Recall that $\om=e^{\frac{2\pi i}{\go}}$. Consider $\Sigma_\go$ defined in \eqref{eq:Sigma prime}.
For any $(\ell,m)\in \Sigma_\go$, by Lemma \ref{lem:HW basis prime} any $z\in\Om_\go$ is an eigenvalue of $X^\ell Z^m$ and we denote by $e^{\ell,m}_{z}$ the corresponding unit eigenvector.
For any vector $\vec{\om}\in \Om_\go^{(\go+1)n}$ of the form (noting that $|\Sigma_\go|=\go+1$)
\begin{equation}\label{eq:defn of vec omega}
	\vec{\om}=({\vec{\om}}^{\ell,m})_{(\ell,m)\in \Sigma_\go},
	\qquad {\vec{\om}}^{\ell,m}=(\omega^{\ell,m}_1,\dots, \omega^{\ell,m}_n)\in \Om_\go^{n},
\end{equation}
we consider the matrix 
\[
\rho(\vec{\om}):=\rho_{1}(\vec{\om})\otimes \cdots\otimes \rho_{n}(\vec{\om})
\]
where 
\[
\rho_k(\vec{\om}):=\frac{1}{\go+1}\sum_{(\ell,m)\in \Sigma_\go}\ketbra{e^{\ell,m}_{\om^{\ell,m}_k}}{e^{\ell,m}_{\om^{\ell,m}_k}}.
\]
Then each $\rho_{k}(\vec{\om})$ is a density matrix and so is $\rho(\vec{\om})$. 

Fix $(\ell,m)\in \Sigma_\go$ and $1\le k\le \go-1$. We have by Lemma \ref{lem:HW basis general}
\begin{align*}
	\tr[X^{k\ell}Z^{k m}\ketbra{e^{\ell,m}_{z}}{e^{\ell,m}_{z}}]
	&=\om^{-\frac{1}{2}k(k-1)\ell m}\langle  e^{\ell,m}_{z},(X^\ell Z^m)^k e^{\ell,m}_{z}\rangle\\
	&=\om^{-\frac{1}{2}k(k-1)\ell m}z^k, \qquad z\in \Om_\go.
\end{align*}
On the other hand, for any $(\ell,m)\neq (\ell',m')\in \Sigma_\go$, we have $(k\ell,km)\notin\langle (\ell',m')\rangle$ by Lemma \ref{lem:HW basis prime}.
From our choice $\gcd (\ell',m')=1$. So Lemma \ref{lem:HW basis general} gives 
\begin{equation*}
	X^{k\ell}Z^{km}X^{\ell'}Z^{m'}=\om^{k\ell' m-k\ell m'}X^{\ell'}Z^{m'}X^{k\ell}Z^{km}
\end{equation*}
with $\om^{k\ell' m-k\ell m'}\neq 1$. This, together with Lemma \ref{lem:orthogonal}, implies
\begin{equation*}
	\tr[X^{k\ell}Z^{km}\ketbra{e^{\ell',m'}_{z}}{e^{\ell',m'}_{z}}]
	=\langle e^{\ell',m'}_{z},X^{k\ell}Z^{km}e^{\ell',m'}_{z}\rangle=0, \qquad z\in \Om_\go.
\end{equation*}
for any $1\le k\le \go-1$.
All combined, for all $1\le k\le \go-1,(\ell,m)\in \Sigma_\go$ and $1\le j\le n$ we get
\begin{align*}
	\tr[X^{k\ell} Z^{k m } \rho_j(\vec{\om})]
	&=\frac{1}{\go+1}\sum_{(\ell',m')\in \Sigma_\go}\left\langle e^{\ell',m'}_{\om^{\ell',m'}_j}, X^{k\ell}Z^{ km}e^{\ell',m'}_{\om^{\ell',m'}_j}\right\rangle\\
	&=\frac{1}{\go+1}\left\langle e^{\ell,m}_{\om^{\ell,m}_j}, X^{k\ell}Z^{ km}e^{\ell,m}_{\om^{\ell,m}_j}\right\rangle\\
	&=\frac{1}{\go+1}\om^{-\frac{1}{2}k(k-1)\ell m}(\om^{\ell,m}_j)^k.
\end{align*}

Note that by Lemma \ref{lem:HW basis prime} any polynomial in $M_\go(\C)^{\otimes n}$ of degree at most $d$ is a linear combination of monomials
\[
A(\vec{k},\vec{\ell},\vec{m};\vec{i}):=
\cdots \otimes X^{k_1\ell_1}Z^{k_1 m_1}\otimes \cdots \otimes X^{k_\kappa\ell_\kappa}Z^{k_\kappa m_\kappa}\otimes\cdots
\]
where
\begin{itemize}
	\item $\vec{k}=(k_1,\dots,k_\kappa)\in \{1,\dots, \go-1\}^\kappa$ with $0\le \sum_{j=1}^{\kappa}k_j\le d$;
	\item $\vec{\ell}=(\ell_1,\dots, \ell_\kappa),\vec{m}=(m_1,\dots, m_\kappa)$ with each $(\ell_j,m_j)\in \Sigma_\go$;
	\item $\vec{i}=(i_1,\dots, i_\kappa)$ with $1\le i_1 <\cdots<i_\kappa\le n$;
	\item $X^{k_j\ell_j}Z^{k_j m_j}$ appears in the $i_j$-th place, $1\le j\le \kappa$, and all the other $n-\kappa$ elements in the tensor product are the identity matrices $\un$.
\end{itemize}
So for any $\vec{\om}\in \Om_\go^{(\go+1)n}$ of the form \eqref{eq:defn of vec omega} we have from the above discussion that
\begin{align*}
	\tr[A(\vec{k},\vec{\ell},\vec{m};\vec{i})\rho(\vec{\om})]
	&=\prod_{j=1}^{\kappa}\tr[X^{k_j\ell_j}Z^{k_j m_j}\rho_{i_j}(\vec{\om})]\\
	&=\frac{\om^{-\frac{1}{2}\sum_{j=1}^{\kappa}k_j(k_j-1)\ell_j m_j}}{(\go+1)^{\kappa}}(\om^{\ell_1,m_1}_{i_1})^{k_1}\cdots (\om^{\ell_\kappa,m_\kappa}_{i_\kappa})^{k_\kappa}.
\end{align*}
Thus $\vec{\om}\mapsto 	\tr[A(\vec{k},\vec{\ell},\vec{m};\vec{i})\rho(\vec{\om})]$ is a monomial on $(\Om_\go)^{(\go+1)n}$ of degree at most $\sum_{j=1}^{\kappa}k_j\le d$.

Now for general polynomial $A\in M_\go(\C)^{\otimes n}$ of degree at most $d$:
\begin{equation*}
	A=\sum_{\vec{k},\vec{\ell},\vec{m},\vec{i}} c(\vec{k},\vec{\ell},\vec{m};\vec{i})A(\vec{k},\vec{\ell},\vec{m};\vec{i})
\end{equation*}
where the sum runs over the above $(\vec{k},\vec{\ell},\vec{m};\vec{i})$. This is the Fourier expansion of $A$ and each $c(\vec{k},\vec{\ell},\vec{m};\vec{i})\in \C$ is the Fourier coefficient. So 
\begin{equation*}
	\|\widehat{A}\|_p=\left(\sum_{\vec{k},\vec{\ell},\vec{m},\vec{i}}|c(\vec{k},\vec{\ell},\vec{m};\vec{i})|^p\right)^{1/p}.
\end{equation*}
To each $A$ we assign the function $f_A$ on $\Om_\go^{(\go+1)n}$ given by
\begin{align*}
	f_A(\vec{\om})&=\tr[A\rho(\vec{\om})]\\
	&=\sum_{\vec{k},\vec{\ell},\vec{m},\vec{i}} \frac{\om^{-\frac{1}{2}\sum_{j=1}^{\kappa}k_j(k_j-1)\ell_j m_j}c(\vec{k},\vec{\ell},\vec{m};\vec{i})}{(\go+1)^{\kappa}}(\om^{\ell_1,m_1}_{i_1})^{k_1}\cdots (\om^{\ell_\kappa,m_\kappa}_{i_\kappa})^{k_\kappa}.
\end{align*}
Note that this is the Fourier expansion of $f_A$ since the monomials $(\om^{\ell_1,m_1}_{i_1})^{k_1}\cdots (\om^{\ell_\kappa,m_\kappa}_{i_\kappa})^{k_\kappa}$ differ for distinct $(\vec{k},\vec{\ell},\vec{m};\vec{i})$'s. Therefore, for $p>0$
\begin{align*}
	\|\widehat{f_A}\|_p
	&=\left(\sum_{\vec{k},\vec{\ell},\vec{m},\vec{i}}\left|\frac{c(\vec{k},\vec{\ell},\vec{m};\vec{i})}{(\go+1)^{\kappa}}\right|^p\right)^{1/p}\\
	&\ge \frac{1}{(\go+1)^{d}}\left(\sum_{\vec{k},\vec{\ell},\vec{m},\vec{i}}|c(\vec{k},\vec{\ell},\vec{m};\vec{i})|^p\right)^{1/p}\\
	&=\frac{1}{(\go+1)^{d}}	\|\widehat{A}\|_p.
\end{align*}
According to Theorem \ref{thm:bh cyclic}, one has
\begin{equation*}
	\|\widehat{f_A}\|_{\frac{2d}{d+1}}\le \textnormal{BH}^{\le d}_{\Omega_\go}\|f_A\|_{\Omega_\go^{(\go+1)n}}
\end{equation*}
for some $\textnormal{BH}^{\le d}_{\Omega_\go}<\infty$. Since each $\rho(\vec{\omega})$ is a density matrix, we have by duality that 
\begin{equation*}
\|f_A\|_{\Omega_\go^{(\go+1)n}}=\sup_{\vec{\omega}\in (\Omega_{\go})^{(\go+1)n}}|\tr[A\rho(\vec{\omega})]|\le \|A\|_{\textnormal{op}}.
\end{equation*}
All combined, we obtain
\begin{equation*}
	\|\widehat{A}\|_{\frac{2d}{d+1}}
	\le (\go+1)^d	\|\widehat{f_A}\|_{\frac{2d}{d+1}}
	\le (\go+1)^d\textnormal{BH}^{\le d}_{\Omega_\go}\|f_A\|_{\Omega_\go^{(\go+1)n}}
	\le (\go+1)^d\textnormal{BH}^{\le d}_{\Omega_\go}\|A\|_{\textnormal{op}}\,. \qedhere
\end{equation*}
\end{proof}

\subsection{The non-prime $\go$ case}

This subsection is devoted to the proof of Theorem \ref{thm:bh HW nonprime}.
 Throughout this part, $\go\ge 4$ is a non-prime integer.

%
%
%
%

%
%

 We start with a substitute of $\Sigma_\go$ in \eqref{eq:Sigma prime} for non-prime $\go$.

\begin{lem}\label{lem:HW basis non-prime}
Fix non-prime $\go\ge 4$. Consider
\begin{equation}\label{eq:Sigma non-prime}
\Sigma_{\go}:=\left\{(\ell,m)\in \Z_\go\times \Z_\go: \gcd(\ell,m)=1\right\}.
\end{equation}
Then $\Z_{\go}\times \Z_{\go}$ is the union of subgroups generated by elements in $\Sigma_\go$:
\begin{equation}
\Z_{\go}\times \Z_{\go}=\bigcup_{(\ell,m)\in \Sigma_\go}\langle (\ell,m)\rangle.
\end{equation}
\end{lem}

\begin{proof}
The proof is direct: any $(\ell,m)\in \Z_{\go}\times \Z_{\go}$ belongs to $\langle(\ell_1, m_1)\rangle$ for $(\ell_1, m_1)=(\ell/\gcd(\ell,m),m/\gcd(\ell,m))\in \Sigma_{\go}$.
\end{proof}

Recall that when $\go$ is prime, for two different subgroups $\langle (\ell_1,m_1)\rangle\neq \langle (\ell_2,m_2)\rangle$ one has the singleton set $\{(0,0)\}$ as their intersection. However, this is no longer the case when $\go$ is not prime. For example, for $\go=6$, we have $\langle (1,0)\rangle\neq \langle (2,3)\rangle$ while $\langle (1,0)\rangle \cap \langle (2,3)\rangle=\{(0,0),(2,0),(4,0)\}$. This difference will make the proof of Theorem \ref{thm:bh HW nonprime} more involved.

\begin{proof}[Proof of Theorem \ref{thm:bh HW nonprime}]
Fix non-prime $\go\ge 4$.  Consider $\Sigma_{\go}$
defined in \eqref{eq:Sigma non-prime}.
Then we know by Lemma \ref{lem:HW basis non-prime} that the set of eigenvalues of $X^\ell Z^m$ is either $\Omega_\go$ or $\Omega_{2\go}\setminus \Omega_{\go}$. In either case, 
suppose that $z$ is an eigenvalue of $X^\ell Z^m$. We denote by $e^{\ell,m}_{z}$ the unit eigenvector of $X^\ell Z^m$ corresponding to $z$. 

 Write 
$\Sigma_\go=\Sigma_\go^{+}\cup \Sigma_\go^{-}$ with
\begin{equation}
\Sigma_\go^{+}:=\{(\ell,m)\in \Sigma_{\go}: \textnormal{ the set of eigenvalues of  $X^\ell Z^m$ is } \Omega_\go\}
\end{equation}
and
\begin{equation}
\Sigma_\go^{-}:=\{(\ell,m)\in \Sigma_{\go}: \textnormal{ the set of eigenvalues of  $X^\ell Z^m$ is } \Omega_{2\go}\setminus\Omega_\go\}.
\end{equation}
As before, for any $\vec{\omega}\in \Omega_{\go}^{|\Sigma_{\go}|n}$ of the form 
\begin{equation}
\vec{\omega}=(\vec{\omega}^{\ell,m})_{(\ell,m)\in \Sigma_{\go}},\qquad
\vec{\omega}^{\ell,m}=(\omega^{\ell,m}_{1},\dots, \omega^{\ell,m}_{n})\in \Omega_{\go}^{n},
\end{equation}
we shall consider 
\begin{equation}
\rho(\vec{\omega}):=\rho_1(\vec{\omega})\otimes \cdots \otimes \rho_n(\vec{\omega})
\end{equation}
where each $\rho_j(\vec{\omega})$ is the average of some eigen-projections of $X^\ell Z^m, (\ell,m)\in \Sigma_\go$. If $(\ell,m)\in \Sigma^{+}_{\go}$, then $X^\ell Z^m$ has $\omega^{\ell,m}_j\in \Omega_{\go}$ as an eigenvalue with $e^{\ell,m}_{\omega^{\ell,m}_j}$ being the unit eigenvector. If $(\ell,m)\in \Sigma^{-}_{\go}$, then $\omega^{\ell,m}_j\in \Omega_{\go}$  is not an eigenvalue of $X^\ell Z^m$. In this case, $X^\ell Z^m$ has $\omega^{1/2}\omega^{\ell,m}_j\in \Omega_{2\go}\setminus\Omega_\go$ as an eigenvalue with $e^{\ell,m}_{\omega^{1/2}\omega^{\ell,m}_j}$ being the unit eigenvector.

For each $1\le j\le n$, consider 
\begin{align}
\rho_j(\vec{\omega})
:=\frac{1}{|\Sigma_{\go}|}\sum_{(\ell,m)\in \Sigma_\go^{+}}\ket{e^{\ell,m}_{\omega^{\ell,m}_j}}\bra{e^{\ell,m}_{\omega^{\ell,m}_j}}+\frac{1}{|\Sigma_{\go}|}\sum_{(\ell,m)\in \Sigma_\go^{-}}\ket{e^{\ell,m}_{\omega^{1/2}\omega^{\ell,m}_j}}\bra{e^{\ell,m}_{\omega^{1/2}\omega^{\ell,m}_j}}.
\end{align}
By definition, each $\rho_j(\vec{\omega})$ is a density matrix and so is $\rho(\vec{\omega}).$

For any $(0,0)\neq (\ell',m')\in \Z_{\go}\times \Z_{\go}$ and any $(\ell,m)\in\Sigma_{\go}$, either $(\ell',m')\notin \langle (\ell,m)\rangle$ or $(\ell',m')=(k\ell,km)$ for some $k\in \Z_{\go}.$ If $(\ell',m')\notin \langle (\ell,m)\rangle$, then by  Lemma \ref{lem:HW basis general}
\begin{equation}
X^{\ell'}Z^{m'}\cdot X^{\ell}Z^{m}=\omega^{\ell m'-\ell' m}X^{\ell}Z^{m}\cdot X^{\ell'}Z^{m'}
\end{equation}
with $\omega^{\ell m'-\ell' m}\neq 1$. So Lemma \ref{lem:orthogonal} gives
\begin{equation}
\tr[ X^{\ell'}Z^{m'} \ket{e^{\ell,m}_{z}}\bra{e^{\ell,m}_{z}}]
=0,
\end{equation}
for any eigenvalue $z$ of $X^\ell Z^m$.

If $(\ell',m')=(k\ell,km)$ for some $k\in \Z_{\go}$, then by Lemma \ref{lem:HW basis general}
\begin{equation}
X^{\ell'}Z^{m'}
=X^{k\ell}Z^{km}
=\omega^{-\frac{1}{2}k(k-1)\ell m}(X^{\ell}Z^{m})^k.
\end{equation} 
So for any eigenvalue $z$ of $X^\ell Z^m$.
\begin{equation}
\tr[ X^{\ell'}Z^{m'} \ket{e^{\ell,m}_{z}}\bra{e^{\ell,m}_{z}}]
=\omega^{-\frac{1}{2}k(k-1)\ell m} z^k.
\end{equation}
All combined, we have for any $\vec{\omega}\in (\Omega_{\go})^{|\Sigma_{\go}|n}$ that 
\begin{align*}
\tr\left[X^{\ell'}Z^{m'}\rho_j(\vec{\omega})\right]
&=\frac{1}{|\Sigma_{\go}|}\sum_{(\ell,m)\in \Sigma_{\go}^{+}:(\ell',m')=(k_{\ell,m}\ell,k_{\ell,m}m)}\omega^{-\frac{1}{2}k_{\ell,m}(k_{\ell,m}-1)\ell m}(\omega_j^{\ell,m})^{k_{\ell,m}}\\
&\quad +\frac{1}{|\Sigma_{\go}|}\sum_{(\ell,m)\in \Sigma_{\go}^{-}:(\ell',m')=(k_{\ell,m}\ell,k_{\ell,m}m)}\omega^{-\frac{1}{2}k_{\ell,m}(k_{\ell,m}-1)\ell m}(\omega^{1/2}\omega_j^{\ell,m})^{k_{\ell,m}}\\
&=\frac{1}{|\Sigma_{\go}|}\sum_{(\ell,m)\in \Sigma_{\go}^{+}:(\ell',m')=(k_{\ell,m}\ell,k_{\ell,m}m)}\omega^{-\frac{1}{2}k_{\ell,m}(k_{\ell,m}-1)\ell m}(\omega_j^{\ell,m})^{k_{\ell,m}}\\
&\quad +\frac{1}{|\Sigma_{\go}|}\sum_{(\ell,m)\in \Sigma_{\go}^{-}:(\ell',m')=(k_{\ell,m}\ell,k_{\ell,m}m)} \omega^{-\frac{1}{2}k_{\ell,m}(k_{\ell,m}-1)\ell m+\frac{1}{2}k_{\ell,m}}(\omega_j^{\ell,m})^{k_{\ell,m}}.
\end{align*}
Here in the summation, when $(\ell',m')\in \langle (\ell,m)\rangle$ we write $(\ell',m')=(k_{\ell,m}\ell,k_{\ell,m}m)$ with $1\le k_{\ell,m}\le \go-1$.
So $\vec{\omega}\mapsto \tr\left[X^{\ell'}Z^{m'}\rho_j(\vec{\omega})\right]$ is a polynomial on $\Omega_{\go}^{|\Sigma_{\go}|n}$ of degree at most $\go-1$, and all the (non-zero) coefficients are of modulus $|\Sigma_{\go}|^{-1}$.
To compare, in the prime $\go$ case, we have only one non-zero term in the above summation. In the non-prime $\go$ case, there might be more than one term. That is, 
we may have different $(\ell_1,m_1)$ and $(\ell_2, m_2)$ in $\Sigma_{\go}$ such that $(\ell',m')= (k_1\ell_1,k_1 m_1)=(k_2 \ell_2, k_2 m_2)$ with $1\le k_1,k_2\le \go-1$. For example for $\go=6$, we have $(0,3)=(3\cdot 0,3\cdot 1)=(3\cdot 2,3\cdot 3)=(3\cdot 4, 3\cdot 5)=(3\cdot 2,3\cdot 1)=(3\cdot 4, 3\cdot 3)$. Though $\tr[X^{\ell'}Z^{m'}\rho_j(\vec{\omega})]$ is no longer a monomial of degree $\deg(X^{\ell'}Z^{m'})$, it is still a non-zero polynomial of degree at most $\go-1$.

Now for any monomial in $M_{\go}(\C)^{\otimes n}$ of degree at most $d$ admitting the form
\begin{equation}
A(\vec{\ell},\vec{m};\vec{i}):=\cdots \otimes X^{\ell_1}Z^{m_1}\otimes \cdots \otimes X^{\ell_\kappa}Z^{m_{\kappa}}\otimes \cdots
\end{equation}
where $\kappa\le d$ and 
\begin{itemize}
\item $\vec{\ell}=(\ell_1,\dots, \ell_{\kappa}), \vec{m}=(m_1,\dots, m_{\kappa})$ with each $(0,0)\neq (\ell_j,m_j)\in \Z_{\go}\times \Z_{\go}$;
\item $\vec{i}=(i_1,\dots, i_{\kappa})$ with $1\le i_1<\cdots <i_{\kappa}\le n$;
\item and each $X^{\ell_j}Z^{m_j}$ appears in the $i_j$-th place, and all the other $n-\kappa$ elements in the tensor product are the identity matrices $\un$. 
\end{itemize}

According to our previous discussion, 
\begin{equation}
\tr[A(\vec{\ell},\vec{m};\vec{i})\rho(\vec{\omega})]
=\prod_{1\le j\le \kappa}\tr[X^{\ell_j}Z^{m_j}\rho_{i_j}(\vec{\omega})]
\end{equation}
is a linear combination of monomials 
$$
(\omega^{a_1,b_1}_{i_1})^{c_1}\cdots (\omega^{a_\kappa,b_\kappa}_{i_\kappa})^{c_\kappa}
$$
of degree at most $(\go-1)\kappa\le (\go-1)d$, with $(a_j,b_j)\in \Sigma_{\go}$ and $1\le c_j\le \go-1$ such that $(c_j a_j, c_j b_j)\equiv (\ell_j,m_j)\mod \go$. This implies that these monomials remember the profile $(\vec{\ell},\vec{m};\vec{i})$ well; \emph{i.e.}, for distinct $(\vec{\ell},\vec{m};\vec{i})\neq (\vec{\ell'},\vec{m'};\vec{i'})$, the corresponding polynomials $\tr[A(\vec{\ell},\vec{m};\vec{i})\rho(\vec{\omega})]$ and $\tr[A(\vec{\ell'},\vec{m'};\vec{i'})\rho(\vec{\omega})]$ do not admit common monomials. Moreover, the coefficients of the those monomials are all of the modulus $|\Sigma_{\go}|^{-\kappa}\ge |\Sigma_{\go}|^{-d}$.
 
For general $A\in M_{\go}(\C)^{\otimes n}$ of degree at most $d$ admitting 
\begin{equation}
A=\sum_{\vec{\ell},\vec{m},\vec{i}}c(\vec{\ell},\vec{m};\vec{i}) A(\vec{\ell},\vec{m};\vec{i})
\end{equation}
 as the Fourier expansion, consider the polynomial
\begin{equation}
f_A(\vec{\omega})=\tr[A\rho(\vec{\omega})]
=\sum_{\vec{\ell},\vec{m},\vec{i}}c(\vec{\ell},\vec{m};\vec{i}) \tr[A(\vec{\ell},\vec{m};\vec{i})\rho(\vec{\omega})]
\end{equation} 
on $\Omega_{\go}^{|\Sigma_\go|n}$. From the above discussion, the $\ell^p$-norm $\|\widehat{f_A}\|_p$ of Fourier coefficients of $f_A$ satisfies
\begin{align*}
\|\widehat{f_A}\|_p\ge 
|\Sigma_{\go}|^{-d}\left(\sum_{\vec{\ell},\vec{m},\vec{i}}|c(\vec{\ell},\vec{m};\vec{i})|^p\right)^{1/p}
=|\Sigma_{\go}|^{-d}\|\widehat{A}\|_p,\qquad p>0.
\end{align*}
Moreover, $f_A$ is of degree at most $(\go-1)d$. So Theorem \ref{thm:bh cyclic} implies 
\begin{equation*}
\|\widehat{f_A}\|_{\frac{2(\go-1)d}{(\go-1)d+1}}\le \textnormal{BH}^{\le (\go+1)d}_{\Omega_\go}\|f_A\|_\infty.
\end{equation*}
Recall that each $\rho(\vec{\omega})$ is a density matrix, so by duality
\begin{equation*}
\|f_A\|_{\infty}=\sup_{\vec{\omega}\in (\Omega_{\go})^{(\go+1)n}}|\tr[A\rho(\vec{\omega})]|\le \|A\|_{\textnormal{op}}.
\end{equation*}
All combined, we prove that 
\begin{equation*}
\|\widehat{A}\|_{\frac{2(\go-1)d}{(\go-1)d+1}}
\le |\Sigma_{\go}|^{d}\|\widehat{f_A}\|_{\frac{2(\go-1)d}{(\go-1)d+1}}
\le |\Sigma_{\go}|^{d}  \textnormal{BH}^{\le (\go+1)d}_{\Omega_\go}\|f_A\|_{\infty}
\le |\Sigma_{\go}|^{d}  \textnormal{BH}^{\le (\go+1)d}_{\Omega_\go}\|A\|_{\mathrm{op}}.
\qedhere
\end{equation*}

\end{proof}

\section{Applications to learning}
\label{sec:learning}
We now apply our new BH inequalities to obtain learning results for observables on qudits.
As with the qubit case, this result may be extended in a certain sense to quantum observables of arbitrary complexity \cite{HCP22}.

For clarity we first extract the main approximation idea of Eskenazis--Ivanisvili \cite{EI22} and include a proof for completeness.
We will need it for vectors in $\C$ rather than $\R$ as it appeared originally but the proof is essentially identical.

\begin{theorem}[Generic Eskenazis--Ivanisvili]
\label{thm:generic-EI}
Let $d\in \mathbb{N}$ and $\eta, B> 0$.
Suppose $v,w\in \C^n$ with $\|v-w\|_\infty\leq \eta$ and $\|v\|_{\frac{2d}{d+1}}\leq B$.
Then for $\widetilde{w}\in \C^n$ defined as $\widetilde{w}_j= w_j \mathbbm{1}_{[|w_j| \geq \eta(1+\sqrt{d+1})]}$
we have the bound
\[\|\widetilde{w}-v\|_{2}^2\leq (e^5\eta^2 d B^{2d})^{\frac{1}{d+1}}.\]
\end{theorem}

\begin{proof}
Let $t>0$ be a threshold parameter to be chosen later.
Define $S_t = \{j: |w_j|\geq t\}$ and note from the triangle inequality that
\begin{align}
    |v_j|\geq |w_j|-|v_j-w_j| = t-\eta & \text{ for } j\in S_t
    \label{eq:big-coords}\\
    |v_j|\leq |w_j|+|v_j-w_j| = t+\eta & \text{ for } j\not\in S_t.
    \label{eq:small-coords}
\end{align}
We may also estimate $|S_t|$ as
\begin{equation}
\label{eq:big-coords-cardinality}
|S_t| = \sum_{j\in S_t}1\overset{\eqref{eq:big-coords}}{\leq} (t-\eta)^{-\frac{2d}{d+1}}\sum_{j\in S_t}|v_j|^{\frac{2d}{d+1}}\leq (t-\eta)^{-\frac{2d}{d+1}}\|v\|_{\frac{2d}{d+1}}^{\frac{2d}{d+1}}\leq (t-\eta)^{-\frac{2d}{d+1}}B^{\frac{2d}{d+1}}.
\end{equation}
With $\widetilde{w}^{(t)} := (w_j\mathbbm{1}_{[|w_j|\geq t]})_{j=1}^n$, we find
\begin{align*}
\|\widetilde{w}^{(t)}-v\|_2^2&=\sum_{j\in S_t} |w_j-v_j|^2 + \sum_{j\not\in S_t}|v_j|^2 \overset{\eqref{eq:small-coords}}{\leq} |S_t|\eta^2 + (t+\eta)^{\frac{2}{d+1}}\sum_{j\in[n]}|v_j|^{\frac{2d}{d+1}}\\
&\overset{\eqref{eq:big-coords-cardinality}}{\leq} B^{\frac{2d}{d+1}}\left(\eta^2(t-\eta)^{-\frac{2d}{d+1}}+(t+\eta)^{\frac{2}{d+1}}\right).
\end{align*}
Choosing $t=\eta(1+\sqrt{d+1})$ then yields $\widetilde{w}$ with error, after some careful scalar estimates,
\[
\|\widetilde{w}-v\|_2^2\leq (e^5\eta^2 d B^{2d})^{\frac{1}{d+1}}.
\]
See \cite[Eqs. 18 \& 19]{EI22} for details on the scalar estimates.
\end{proof}

In the context of low-degree learning, $v$ is the true vector of Fourier coefficients, and $w$ is the vector of empirical coefficients obtained through Fourier sampling. 
\subsection{Qudit learning}
We first pursue a learning algorithm that finds a (normalized) $L_2$ approximation to a low-degree operator $\cA$.
Then we'll see how this extends to an algorithm finding an approximation $\widetilde{\cA}$ with good mean-squared error over certain distributions of states for target operators $\cA$ of any degree.

For clarity and brevity we will assume that for a mixed state $\rho$ the quantity $\tr[\cA\rho]$ can be directly computed.
Of course this is not true in practice; in the lab one must take many copies of $\rho$, collect observations $m_1,\ldots, m_s$ and form the estimate $\frac{1}{s}\sum_jm_j\approx \tr[\cA\rho]$.
The analysis required to relax these assumptions from the following results are routine so we omit them.

\subsubsection{Low-degree Qudit learning}
\begin{theorem}[Low-degree Qudit Learning]
    \label{thm:qudit-ld-learning}
 Fix $\epsilon,\delta,t>0$ and $d\ge 1$. Then there is a collection $S$ of product states such that for any quantum observable $\cA\in M_\go(\C)^{\otimes n}$ of degree at most $d$ with $\|\cA\|_\textnormal{op}\leq t$, and a number
    \[\mathcal{O}\Big(\big( \go t\big)^{C\cdot d^2}d^2\eps^{-d-1}\log\left(\tfrac{n}{\delta}\right)\Big)\]of samples of the form $(\rho, \tr[\cA \rho]),$ $\rho\sim\mathcal{U}(S)$, we may with confidence $1-\delta$ learn an observable $\widetilde{\cA}$ with $\|\cA-\widetilde{\cA}\|_2^2\leq \eps$.
\end{theorem}
Here $\|\cA\|_2$ denotes the normalized $L_2$ norm induced by the inner product $\langle A,B\rangle:=\go^{-n}\tr[A^\dagger B]$.
Also, for applications it is natural to assume $\|\cA\|_\mathrm{op}$ is bounded independent of $n$.
However we choose to include explicit mention of $\|\cA\|_\mathrm{op}$ here as it will be useful later.

\begin{proof}
    We will first pursue an $L_\infty$ estimate of the Fourier coefficients in the Gell-Mann basis.
    To that end, sample $\vec{\boldsymbol{x}}_1,\ldots, \vec{\boldsymbol{x}}_s\overset{\mathsf{iid}}{\sim} \{-1,1\}^{n(\go^2-1)}$.
    As in the proof of Theorem \ref{thm:GM}, for any such $\vec{\boldsymbol x}$ we partition indices as $\vec{\boldsymbol{x}} = (\boldsymbol{x}_1,\ldots, \boldsymbol{x}_n)\in(\{-1,1\}^{\go^2-1})^n$ with each $\boldsymbol{x}_\ell$, $1\leq\ell\leq n$, corresponding to a qudit.
    Each $\boldsymbol{x}_\ell$ is further partitioned as
    \[\boldsymbol{x}_\ell=(x^{(\ell)},y^{(\ell)},z^{(\ell)})\in\{-1,1\}^{\binom{\go}{2}}\times\{-1,1\}^{\binom{\go}{2}}\times\{-1,1\}^{\go-1}\,,\]
    with each sub-coordinate associated with a specific Gell-Mann basis element for that qudit.

    Again for each $\vec{\boldsymbol{x}}$, for each qudit $\ell\in[n]$ form the mixed state
    \[\rho(x^{(\ell)},y^{(\ell)},z^{(\ell)})=\frac{1}{3\binom{\go}{2}}\left(\sum_{1\le j<k\le \go}  A_{jk}^{(x^{(\ell)}_{jk})} + \sum_{1\le j<k\le \go} B_{jk}^{(y^{(\ell)}_{jk})} + \sum_{m=1}^{\go-1} z^{(\ell)}_m \tfrac{1}{\sqrt{2\go}}\bC_m+ \tfrac{\go-1}{2}\cdot  \un\right).\]
    Then we may define for $\vec{\boldsymbol{x}}$ the $n$ qudit mixed state
    \[\rho(\vec{\boldsymbol{x}}) = \bigotimes_{\ell=1}^{n}\rho(x^{(\ell)},y^{(\ell)},z^{(\ell)})\]
    and consider the function
    \[f_\cA(\vec{\boldsymbol{x}}):=\tr[\cA \cdot \rho(\vec{\boldsymbol{x}})].\]
    Let $S(\alpha)$ denote the index map from the GM basis to subsets of $[n(\go^2-1)]$.
    With these states in hand and in view of the identity
    \[\widehat{f_{\cA}}\big(S(\alpha)\big)=c^{|\alpha|}\widehat{\cA}(\alpha) \quad \text{with} \quad c:= \frac{\sqrt{\go/2}}{3\binom{\go}{2}}\;<1,\] we may now define the empirical Fourier coefficients
    \[\cW(\alpha) = c^{-|\alpha|}\cdot\frac{1}{s}\sum_{t=1}^sf_{\cA}(\vec{\boldsymbol{x}}_t)\prod_{j\in S(\alpha)}x_j \;=\; c^{-|\alpha|}\frac{1}{s}\sum_{t=1}^s\tr[\cA\cdot \rho(\vec{\boldsymbol{x}}_t)]\prod_{j\in S(\alpha)}x_j.\]
    The coefficient $\cW(\alpha)$ is a sum of bounded i.i.d. random variables each with expectation $\widehat{\cA}(\alpha)$, so by Hoeffding's inequality \cite[Theorem 2.2.6]{vershynin} we have 
    \[\Pr\big[|\cW(\alpha)-\widehat{\cA}(\alpha)|\geq \eta\big]\leq 2\exp(-s\eta^2c^{2|\alpha|})\,.\]
    Taking the union bound, we find the chance of achieving $\ell_\infty$ error $\eta$ is:
    \[\Pr\big[|\cW(\alpha)-\widehat{\cA}(\alpha)|<\eta\text{ for all $\alpha$ with } |\alpha|\leq d\big] \geq 1-2\sum_{k=0}^d\binom{n}{k}\exp(-s\eta^2c^{2d})\,,\]
    which again we shall require to be $\geq 1-\delta$. 

    Applying Theorem \ref{thm:generic-EI} to obtain $\widetilde{\cW}$ and recalling $\|\widehat{\cA}\|_{\frac{2d}{d+1}}\leq \BH^{\leq d}_{\mathrm{GM}(\go)}\|\cA\|_\text{op}$ we find the estimated operator
    \[\widetilde{\cA}:=\sum_\alpha\widetilde{\cW}(\alpha)M_\alpha\]
    has $L_2$-squared error
    \[\|\widetilde{\cA}-\cA\|_2^2\overset{\text{(Parseval)}}{=}\sum_\alpha\big|\widetilde{\cW}(\alpha)-\widehat{\cA}(\alpha)\big|^2\leq \big(e^5\eta^2d(\BH^{\leq d}_{\mathrm{GM}(\go)}\|\cA\|_\text{op})^{2d}\big)^{\frac{1}{d+1}}\,.\]
    Thus to obtain error $\leq \eps$ it suffices to pick $\eta^2= \eps^{d+1}e^{-5}d^{-1}(\BH^{\leq d}_{\mathrm{GM}(\go)}\|\cA\|_\text{op})^{-2d}$, which entails by standard estimates that the algorithm will meet the requirements with a sample count of 
    \[s\geq e^6d^2\big(18\go ^3\BH_{\mathrm{GM}(\go)}^{\leq d}\|\cA\|_\text{op}\big)^{2d}\log\left(\tfrac{2en}{\delta}\right)\eps^{-d-1}\,. \qedhere\]
\end{proof}

\subsubsection{Learning arbitrary qudit observables}
As observed by Huang, Chen, and Preskill in \cite{HCP22}, there are certain distributions $\mu$ of input states $\rho$ for which a low-degree truncation $\widetilde\cA$ of any observable $\cA$ gives a suitable approximation as measured by $\E_{\rho\sim \mu}|\tr[\widetilde\cA\rho]-\tr[\cA\rho]|^2$.
This observation extends easily to qudits, which in turn ends up generalizing the phenomenon in the context of qubits as well.
\label{sec:learning-arb}
\begin{definition}
    For a 1-qudit unitary $U$ let $U_j:=\mathbf{I}^{\otimes j-1}\otimes U\otimes \mathbf{I}^{\otimes n-j}$.
    Then for a probability distribution $\mu$ on $n$-qudit densities and $j\in[n]$ let $\mathop{\mathrm{Stab}}_j(\mu)$ be the set of $U\in \mathrm{U}(\go)$ such that for all densities $\rho$, 
    \[\mu(\rho) = \mu(U_j\rho U_j^{\dagger})\,.\]
\end{definition}
\begin{remark}
    For a set $S$ of states, define $\mathop{\mathrm{Stab}}_j(S)=\{U\in\mathrm{U}(\go):U_jSU_j^\dagger\subseteq S\}$.
    Then it can be seen easily that $\mathop{\mathrm{Stab}}_j(\mu)$ is equal to the intersection of the stabilizers of the level sets of $\mu$.
    That is, $\mathop{\mathrm{Stab}}_j(\mu)=\bigcap_{0\leq r\leq 1}\mathop{\mathrm{Stab}}_j\big(\mu^{-1}(r)\big)$.
\end{remark}

Recall the definition of a unitary $t$-design \cite{DCEL09}.

\begin{definition}
    Consider $P_{t,t}(U)$, a polynomial of degree at most $t$ in the matrix elements of unitary $U\in\mathrm{U}(\go)$ and of degree at most $t$ in the matrix elements of $U^\dagger$.
    Then a finite subset $S$ of the unitary group $\mathrm{U}(\go)$ is a \emph{unitary t-design} if for all such $P_{t,t}$,
    \[\frac{1}{|S|}\sum_{U\in S}P_{t,t}(U)=\E_{U\sim\mathrm{Haar}(\mathrm{U}(\go))}[P_{t,t}(U)]\,.\]
\end{definition}

We are ready to name the distributions for which low-degree truncation is possible without losing much accuracy.
\begin{definition}
    Call a distribution $\mu$ on $n$-qudit densities \emph{locally 2-design invariant} (L2DI) if for all $j\in[n]$, $\mathrm{Stab}_j(\mu)$ contains a unitary $2$-design.
\end{definition}

Of course, the $n$-fold tensor product of Haar-random qudits is an L2DI distribution, but there are many other possible distributions and in general they can be highly entangled.
When $\go=2$ and the $2$-design leaving $\mu$ locally invariant is the single-qubit Clifford group, such distributions are termed \emph{locally flat} in \cite{HCP22}.
For any prime $\go$ the Clifford group on $\mathcal{H}_\go$ is a 2-design \cite{GAE}.
Importantly, however, any distribution just on classical inputs $\ket{x}, x\in \{0,1,\ldots, \go-1\}^n$ is not L2DI, as a consequence of the following general observation:

\begin{proposition}
    Suppose $\mu$ is an L2DI distribution over pure product states $\rho = \bigotimes_{j=1}^n\ketbra{\psi_j}{\psi_j}$.
    For $j\in[n]$ define $S_j=\{\rho_{\{j\}}:\rho\in\mathrm{supp}(\mu)\}$, where $\rho_{\{j\}}$ denote the marginal state on qudit $j$. 
    Then $|S_j|\geq \go^2$.
\end{proposition}
\begin{proof}
    By the L2DI property we have a 2-design $X\subseteq\mathrm{SU}(\go)$ under which $\mu$ is locally invariant.
    Let $\rho$ be a pure state in $S_j$ and note that $P:=\{U\rho U^\dagger: U\in X\}\subseteq S_j$.
    On the other hand, $P$ forms a complex-projective 2-design because it is the orbit of a unitary 2-design \cite{Dankert, DCEL09}.
    And any complex-projective 2-design in $\go$ dimensions must have cardinality at least $\go^2$ \cite[Theorem 4]{Scott_2006}.
\end{proof}


The truncation theorem for qudits goes through much the same way as it does for locally flat qubit distributions in \cite{HCP22}.
\begin{theorem}
    \label{thm:L2DI-NS}
    Let $\cA$ be an operator on $\mathcal{H}_\go^{\otimes n}$ and $\mu$ a probability distribution on densities.
    Then if $\mu$ is L2DI we have
    \[\E_{\rho\sim\mu}|\tr[\cA\rho]|^2\leq \sum_\alpha\big(\tfrac{\go}{\go^2-1}\big)^{|\alpha|}|\widehat{\cA}(\alpha)|^2\,.\]
\end{theorem}
The reader will notice a marked similarity to the Fourier-basis expression of noise stability $\E_{x\sim_\delta y}f(x)f(y)=\langle f, \mathrm{T}\hspace{-1pt}_\delta f\rangle = \sum_{S\subseteq[n]}\delta^{|S|}\widehat{f}(S)^2$ for Boolean functions (\emph{e.g.}, \cite{Odonnell}).
\begin{proof}
    Expanding $\cA$ in the Gell-Mann basis we obtain
    \begin{align}
    \E_{\rho\sim\mu}|\tr[\cA\rho]|^2 \;=\; \E_\rho \Big|\sum_{\alpha}\widehat{\cA}(\alpha)\tr[M_\alpha\cdot\rho]\Big|^2
    \;=\; \sum_{\alpha,\beta}\overline{\widehat{\cA}(\alpha)}\widehat{\cA}(\beta)\underbrace{\E_\rho\tr[M_\alpha\otimes M_\beta \cdot \rho\otimes\rho]}_{(*)}\,.
    \label{eq:difference-decomp}
    \end{align}
    
    Now for fixed $\alpha,\beta$ we examine the expectation $(*)$.
    For $j\in[n]$ let $D_j$ be a $2$-design in $\mathrm{Stab}_j(\mu)$ guaranteed to exist by the L2DI property and define $D=\bigotimes_{j=1}^nD_j$ in the natural way.
    Then because $D_j$ leaves $\mu$ invariant,
    \begin{align*}
    (*)&=\E_\rho\E_{U\sim\mathcal{U}(D)}\tr[M_\alpha\otimes M_\beta \cdot U\rho U^\dagger\otimes U\rho U^\dagger]\\
    &=\E_\rho\tr\left[\bigotimes_{j=1}^n\E_{U_j\sim \mathcal{U}(D_j)}(U_j^\dagger M_{\alpha_j}U_j)\otimes(U_j^\dagger M_{\beta_j}U_j) \cdot \rho\otimes\rho\right].
    \end{align*}

    For a single coordinate $j$ let us study $(U_j^\dagger M_{\alpha_j}U_j)\otimes(U_j^\dagger M_{\beta_j}U_j)$.
    Certainly if $M_{\alpha_j}=M_{\beta_j}=\mathbf{I}$ then
    \begin{equation}
    \label{eq:identity-case}
        \E_{U_j}[(U_j^\dagger M_{\alpha_j}U_j)\otimes(U_j^\dagger M_{\beta_j}U_j)]=\mathbf{I}\otimes\mathbf{I}\,.
    \end{equation}
    Now suppose at least one of $M_{\alpha_j},M_{\beta_j}\neq \mathbf{I}$.
    Then $\tr[M_{\alpha_j}\otimes M_{\beta_j}]=0$ because non-identity elements of $\mathrm{GM}(\go)$ are traceless.
    Letting $\mathbf{F} := \sum_{j,k=1}^\go\ketbra{jk}{kj}$ denote a generalized \textsf{SWAP} operator, we have by the previous observation and properties of unitary 2-designs (see \emph{e.g.}, \cite[\S II.A]{GAE}) that 
    \begin{align}
        &\E_{U_j}\big[(U_j^\dagger M_{\alpha_j}U_j)\otimes (U_j^\dagger M_{\beta_j}U_j)\big] \\
        =& \frac{\tr[(\mathbf{I}+\mathbf{F})\, M_{\alpha_j}\!\otimes M_{\beta_j}]}{\go^2+\go}\left(\frac{\mathbf{I}+\mathbf{F}}{2}\right) + \frac{\tr[(\mathbf{I}-\mathbf{F})\,M_{\alpha_j}\!\otimes  M_{\beta_j}]}{\go^2-\go}\left(\frac{\mathbf{I}-\mathbf{F}}{2}\right)\nonumber\\
        =& \tr[\mathbf{F}\cdot M_{\alpha_j}\!\otimes M_{\beta_j}]\cdot\frac{1}{\go^2-1}\left(\mathbf{F}-\frac{1}{\go}\mathbf{I}\right)\,.
        \label{eq:2-design}
    \end{align}
    
    Let $S_\alpha = \{j:\alpha_j\neq \mathbf{I}\}$.
    Then combining Eqs. \eqref{eq:identity-case}, \eqref{eq:2-design} with the fact $\tr[\mathbf{F}\cdot M_{\alpha_j}\otimes M_{\beta_j}] = \tr[M_{\alpha_j}M_{\beta_j}] = \go\delta_{{\alpha_j}{\beta_j}}$ we have
    \begin{align}
        \label{eq:star-delta-and-flip}
        (*)&= \delta_{\alpha\beta}\textstyle\left(\frac{\go}{\go^2-1}\right)^{|\alpha|}\E_\rho\tr\big[(\mathbf{F}-\tfrac{1}{\go}\mathbf{I})^{\otimes|\alpha|}\rho_{S_\alpha}\!\otimes\rho_{S_\alpha}\big]\,.
    \end{align}
    Now we claim for any $\rho$ on $m$ qudits,
    \begin{equation}
    \label{eq:F-bound}
        \tr\big[(\mathbf{F}-\tfrac{1}{\go}\mathbf{I})^{\otimes m}\rho\otimes\rho\big]\leq 1\,.
    \end{equation}
    To see this, notice $\mathbf{F}=\frac{1}{\go}\sum_{\alpha}M_\alpha\otimes M_\alpha$, a sum of all GM basis elements doubled-up.
    Hence
    \begin{align*}
        \tr\big[(\mathbf{F}-\tfrac{1}{\go}\mathbf{I})^{\otimes m}\rho\otimes\rho\big] &=\tr\left[{\textstyle \bigotimes_{j=1}^m}\left(\tfrac{1}{\go}{\textstyle\sum_{\beta_j\neq \mathrm{Id}}}\,M_{\beta_j}\otimes M_{\beta_j} \right)\cdot\rho\otimes\rho\right] \\
        &= \frac{1}{\go^n}\sum_{\beta, |\beta|=m}\tr[M_\beta\otimes M_\beta\cdot\rho\otimes\rho]\\
        &\leq \frac{1}{\go^n}\sum_{\text{all}\,\beta}\tr[M_\beta \rho]^2\\
        &= \tr[\rho^2]\;\;\leq\;\; 1\,.
    \end{align*}
    Combining Eqs. \eqref{eq:star-delta-and-flip} and \eqref{eq:F-bound} we in summary have the bound
    \[(*)\leq \delta_{\alpha\beta}\left(\frac{\go}{\go^2-1}\right)^{|\alpha|}\,.\]
    Returning to Eq. \eqref{eq:difference-decomp} we conclude
    \begin{align*}
        \E_{\rho\sim\mu}|\tr[\cA\rho]|^2 \;= \sum_{\alpha} \left(\frac{\go}{\go^2-1}\right)^{-|\alpha|}|\widehat{\cA}(\alpha)|^2\,. & \qedhere
    \end{align*}
\end{proof}

\begin{definition}
    \label{defn:truncation}
    Let $\cA$ be an operator with Gell-Mann decomposition $\cA =\sum_\alpha \widehat{\cA}_\alpha  M_\alpha$.
    Then for $d\in[n]$ define its \emph{degree-$d$ truncation} to be $\cA^{\leq d} = \sum_{\alpha, |\alpha|\leq d}\widehat{\cA}_\alpha M_\alpha$.
\end{definition}
\begin{remark}
The choice of the GM decomposition here is essentially without loss of generality:
consider any orthogonal basis $B$ for $M_\go(\C)$ containing the identity and define $\cA^{\leq d}_B$ analogously to Definition \ref{defn:truncation}, keeping the definition of degree analogous to that for the GM basis too.
Then we have $\cA^{\leq d}_B=\cA^{\leq d}$, as can be seen easily by expanding one basis in the other. 
\end{remark}

\begin{corollary}
\label{cor:low-degree-approx}
$\E_{\rho\sim \mu}|\tr[\cA\rho]-\tr[\cA^{\leq d}\rho]|^2\leq \left(\frac{\go}{\go^2-1}\right)^{d}\|\cA\|_2^2$.
\end{corollary}
\begin{proof}
    We apply Theorem \ref{thm:L2DI-NS} to obtain
    \begin{align*}
    \E_{\rho\sim \mu}|\tr[\cA\rho]-\tr[\cA^{\leq d}\rho]|^2 &= \E_{\rho\sim \mu}|\tr[(\cA-\cA^{\leq d})\rho]|^2\\
    &\leq \sum_{\alpha, |\alpha|>d}\left(\tfrac{\go}{\go^2-1}\right)^{|\alpha|}|\widehat{\cA}(\alpha)|^2 \leq \left(\tfrac{\go}{\go^2-1}\right)^d\|\cA\|_2^2\,.\qedhere
    \end{align*}
\end{proof}

\begin{theorem}
    Let $\cA$ be any observable on $\mathcal{H}_\go^{\otimes n}$, of any degree, with $\|\cA\|_\mathrm{op}\leq 1$.
    Fix an error threshold $\epsilon >0$ and a failure probability $\delta > 0$ and put $t=\log_{\go^2-1}(4/\epsilon)$.
    Then there is a set $S$ of product states such that with a number
    \[s=\mathcal{O}\Big(\log\big(\tfrac{n}{\delta}\big)e^{c\cdot\log^2(\frac{1}{\eps})}(18K^3\|\cA^{\leq t}\|_\mathrm{op})^{2t}\Big)\]
    of samples $(\rho, \tr[\cA \rho])$, $\rho\sim\mathcal{U}(S)$, an approximate operator $\widetilde{\cA}$ may be formed in time $\mathrm{poly}(n)$ with confidence $1-\delta$ satisfying
    \[\E_{\rho\sim\mu}|\tr[\widetilde{\cA}\rho]-\tr[\cA\rho]|^2\leq \eps\]
    for any L2DI distribution $\mu$.
\end{theorem}
\begin{proof}
    Choose the truncation degree $d=\log_{\go^2-1}(4/\eps)$.
    Then the triangle inequality and Corollary \ref{cor:low-degree-approx} give
    \begin{align*}
        \E_{\rho\sim \mu}|\tr[\widetilde{\cA}^{\leq d}\rho]-\tr[\cA\rho]|^2 & \leq 2\E_{\rho\sim \mu}|\tr[{\cA}^{\leq d}\rho]-\tr[\cA\rho]|^2+2\E_{\rho\sim \mu}|\tr[\widetilde{\cA}^{\leq d}\rho]-\tr[\cA^{\leq d}\rho]|^2\\
        &\leq 2\left(\tfrac{1}{\go^2-1}\right)^d+2\|\widetilde{\cA}^{\leq d} - \cA^{\leq d}\|_2^2\\
        &\leq \eps/2 + 2\|\widetilde{\cA}^{\leq d} - \cA^{\leq d}\|_2^2\,.
    \end{align*}
    So we need to choose a number of samples such that with confidence $1-\delta$, the low-degree qudit learning algorithm (Theorem \ref{thm:qudit-ld-learning})  yields a $\cA^{\leq d}$ such that $\|\widetilde{\cA}^{\leq d} - \cA^{\leq d}\|_2^2\leq\eps/4$.
    This requires no more than
    \[C\log\left(\frac{2en}{\delta}\right)e^{C'\log^2(4/\eps)}(18\go^{3}\|\cA^{\leq t}\|_\mathrm{op})^{2t}\]
    samples, where $t=\log_{\go^2-1}(4/\epsilon)$ and $C,C'$ are constants $> 1$.
\end{proof}

This learning theorem may be of interest even in the context of qubits.
In particular, for a small divisor $k$ of $n$, a system of $n$ qubits may be interpreted as $n/k$-many $2^k$-level qudits, and there may be interesting distributions over states in this system which are only L2DI when viewed as ``virtual qudits'' in this way. 

\newcommand{\etalchar}[1]{$^{#1}$}

	\appendix

	\section{Eigenvectors of the Heisenberg--Weyl basis}
		\label{appendix}
	As a complement to Lemmas \ref{lem:HW basis general} and \ref{lem:HW basis prime}, we compute the eigenvectors of $X^\ell Z^m,\gcd(\ell,m)=1$ for readers interested in their concrete forms.

	Let us begin with the simple case of prime $\go$. Assume first that $\ell=0$ and $m\neq 0$, then the eigenvalues of $Z^m$ are
		\[
    \{\om^{jm}:j\in\mathbb{Z}_\go\}=\{\om^{j}:j\in \mathbb{Z}_\go\},
    \]
		since $\go$ is prime. If $\ell\neq 0$, then we may relabel the standard basis $\{e_j:j\in\mathbb{Z}_\go\}$ as $\{e_{j\ell}:j\in\mathbb{Z}_\go\}$. Consider the non-zero vectors 
		\begin{equation}\label{eq:eigenvector prime}
			\zeta_k:=\sum_{j\in \mathbb{Z}_\go}\om^{\frac{1}{2}j(j-1)\ell m-jk}e_{j\ell},\qquad k\in \mathbb{Z}_\go.
		\end{equation}
		Note that the summation in \eqref{eq:eigenvector prime} can be understood as in the sense of mod $\go$. In fact, for any $j'=j+r\go$ for some $r\in \Z$, we have 
\begin{equation*}
\frac{1}{2}j'(j'-1)\ell m-j'k
=\frac{1}{2}j(j-1)\ell m-jk+\frac{1}{2}r\go(r\go+2j-1)\ell m-rk\go
\end{equation*}
A direct computation shows: for all $k\in \mathbb{Z}_\go$
		\begin{align*}
			X^{\ell}Z^m \zeta_k
			&=\sum_{j\in \mathbb{Z}_\go}\om^{\frac{1}{2}j(j-1)\ell m-jk}\cdot \om^{j\ell m}X^{\ell}e_{j\ell}\\
			&=\sum_{j\in \mathbb{Z}_\go}\om^{\frac{1}{2}j(j+1)\ell m-jk}e_{(j+1)\ell}\\
			&=\sum_{j\in \mathbb{Z}_\go}\om^{\frac{1}{2}j(j-1)\ell m-jk+k}e_{j\ell}\\
			&=\om^k\zeta_k.\qedhere
		\end{align*}
		
			\medskip
In general, we have the following result. 
\begin{proposition}
Assume $\go\ge 3$ and $\ell,m\in \Z_\go$ are such that $\gcd(\ell,m)=1$. Let $d:=\gcd(\go, \ell)$ and write $\go=d\go_1$, so that any $k\in \Z_{\go}$ can be uniquely written as $k=s+j\ell$ with $s\in \Z_{d},j\in \Z_{\go_1}$. So the canonical basis of $\C^{\go}$ takes the form
\begin{equation}
(e_k)_{k\in \Z_{\go}}=(e_{s+j\ell})_{s\in \Z_{d},j\in \Z_{\go_1}}.
\end{equation}
\begin{enumerate}
\item If $\go_1$ is odd, then $X^\ell Z^m$ has 
 \begin{equation}\label{eq:eigenvector odd}
 \xi_{s,t}:=\sum_{j\in \Z_{\go_1}}\omega^{\frac{1}{2}j(j-1)\ell m -jdt}e_{s+j\ell}, \qquad (s,t)\in \Z_{d}\times \Z_{\go_1}
 \end{equation}
 as an eigenvector with eigenvalue $\omega^{dt+ms}$. In particular, the set of eigenvalues of $X^\ell Z^m$ is $\Omega_\go.$
\item If $\go_1$ is even, then $X^\ell Z^m$ has
\begin{equation}\label{eq:eigenvector even}
 \eta_{s,t}:=\sum_{j\in \Z_{\go_1}}\omega^{\frac{1}{2}j(j-2)\ell m -jdt}e_{s+j\ell},  \qquad (s,t)\in \Z_{d}\times \Z_{\go_1}
\end{equation}
as an eigenvector with eigenvalue $\omega^{\frac{1}{2}\ell m+dt+ms}$.  In particular, the set of eigenvalues of $X^\ell Z^m$ is  $\Omega_\go$ or $\Omega_{2\go}\setminus\Omega_\go$, depending on $\ell m$ is even or odd. 
\end{enumerate}
\end{proposition}

\begin{proof}
\begin{enumerate}
\item 
We verify that \eqref{eq:eigenvector odd} is well-defined with $\Z_{\go_1}$ understood in the sense of $\mod \go_1$, which will be used later. In fact, for $j'=j+r\go_1$ with $r\in \Z$, 
\begin{align*}
&\frac{1}{2}j'(j'-1)\ell m -j'dt\\
&\qquad=\frac{1}{2}j(j-1)\ell m -jdt+\frac{1}{2}r(2j-1+r\go_1) (\ell/d)m\go-rt\go.
\end{align*}
Since $\go_1$ is odd, $r(2j-1+r\go_1)$ is even no matter $r$ is odd or even.  Thus 
\begin{align*}
\frac{1}{2}j'(j'-1)\ell m -j'dt
\equiv\frac{1}{2}j(j-1)\ell m -jdt\mod \go_1.
\end{align*}
It is clear that 
$$s+j'\ell=s+j\ell +r(\ell/d)\go \equiv s+j\ell \mod \go.$$
This justifies the definition of \eqref{eq:eigenvector odd}.
Therefore, 
\begin{align*}
X^{\ell}Z^{m} \xi_{s,t}
& =\sum_{j\in \Z_{\go_1}}\omega^{\frac{1}{2}j(j-1)\ell m -jdt}\omega^{(s+j\ell)m}e_{s+(j+1)\ell}\\
&=\sum_{j\in \Z_{\go_1}}\omega^{\frac{1}{2}(j-1)(j-2)\ell m -(j-1)dt}\omega^{(s+(j-1)\ell)m}e_{s+j\ell}\\
&=\sum_{j\in \Z_{\go_1}}\omega^{dt+ms}\omega^{\frac{1}{2}j(j-1)\ell m-jdt}e_{s+j\ell}\\
&=\omega^{dt+ms} \xi_{s,t}.
\end{align*}
So for any $(s,t)\in \Z_{d}\times \Z_{\go_1}$, $\omega^{dt+ms}$ is an eigenvalue of $X^{\ell}Z^{m}$ with $\xi_{s,t}$ as an eigenvector. Since $d$ divides $\ell$, and $\gcd (\ell,m)=1$, we have $\gcd(d,m)=1$. Therefore 
\begin{equation}
\{\omega^{dt+ms}: (s,t)\in \Z_{d}\times \Z_{\go_1}\}=\Omega_{\go}.
\end{equation}

\item  We first justify the definition of \eqref{eq:eigenvector even}. If $j'=j+r\go_1$ for some $r\in \Z$, then 
\begin{align*}
&\frac{1}{2}j'(j'-2)\ell m -j'dt\\
&\qquad=\frac{1}{2}j(j-2)\ell m -jdt+(j-1+\frac{1}{2}r\go_1)(\ell/d)rm\go-rt\go.
\end{align*}
Since $\go_1$ is even, we have 
\begin{equation*}
\frac{1}{2}j'(j'-2)\ell m -j'dt\equiv\frac{1}{2}j(j-2)\ell m -jdt\mod \go.
\end{equation*}
Moreover, it is clear that $s+j'\ell \equiv s+j\ell \mod \go$. So  \eqref{eq:eigenvector even} is well-defined. Therefore, 
\begin{align*}
X^{\ell}Z^{m} \eta_{s,t}
& =\sum_{j\in \Z_{\go_1}}\omega^{\frac{1}{2}j(j-2)\ell m -jdt}\omega^{(s+j\ell)m}e_{s+(j+1)\ell}\\
&=\sum_{j\in \Z_{\go_1}}\omega^{\frac{1}{2}(j-1)(j-3)\ell m -(j-1)dt}\omega^{(s+(j-1)\ell)m}e_{s+j\ell}\\
&=\sum_{j\in \Z_{\go_1}}\omega^{\frac{1}{2}\ell m+ dt+ms}\omega^{\frac{1}{2}j(j-1)\ell m-jdt}e_{s+j\ell}\\
&=\omega^{\frac{1}{2}\ell m+dt+ms} \eta_{s,t}.       
\end{align*}
Then the set of eigenvalues of $X^{\ell}Z^m$ is 
\begin{equation}
\left\{\omega^{\frac{1}{2}\ell m+sm+dt}:(s,t)\in \Z_{d}\times \Z_{\go_1}\right\}
=\left\{\omega^{\frac{1}{2}\ell m+k}:k\in \Z_{\go}\right\}
\end{equation}
which is either $\Omega_{\go}$ or $\Omega_{2\go}\setminus \Omega_{\go}$  depending on $\ell m$ is even or not.  \qedhere
\end{enumerate}
\end{proof}


\begin{thebibliography}{DGMSP19}
	
	\bibitem[AA14]{AA}
	Scott Aaronson and Andris Ambainis.
	\newblock The need for structure in quantum speedups.
	\newblock {\em Theory of Computing}, 10(1):133--166, 2014.
	
	\bibitem[AEHK16]{AEHK}
	Ali Asadian, Paul Erker, Marcus Huber, and Claude Kl{\"o}ckl.
	\newblock Heisenberg-{W}eyl observables: {B}loch vectors in phase space.
	\newblock {\em Physical Review A}, 94(1):010301, 2016.
	


    \bibitem[BKSV{\etalchar{+}}23]{BKSVZremez}
 Lars Becker, Ohad Klein, Joseph Slote, Alexander Volberg, and  Haonan Zhang.
\newblock Dimension-free Remez Inequalities and norm designs.
\newblock {\em arXiv: 2310.07926}.

	
	\bibitem[BK08]{BK08GM}
	Reinhold~A Bertlmann and Philipp Krammer.
	\newblock Bloch vectors for qudits.
	\newblock {\em Journal of Physics A: Mathematical and Theoretical},
	41(23):235303, 2008.
	
		\bibitem[Ble01]{Blei}
	Ron Blei.
	\newblock {\em Analysis in integer and fractional dimensions}, volume~71.
	\newblock Cambridge University Press, 2001.
	
	\bibitem[Bou02]{Bou02}
	Jean Bourgain.
	\newblock On the distribution of the {F}ourier spectrum of boolean functions.
	\newblock {\em Israel Journal of Mathematics}, 131(1):269--276, 2002.
	
		\bibitem[BKKKL92]{BGKKL}
Jean Bourgain, Jeff Kahn, Gil Kalai, Yitzhak Katznelson, and Nathan Linial.
\newblock The influence of variables in product spaces.
\newblock {\em Israel Journal of Mathematics}, 77(1-2):55--64, February 1992.

	
	\bibitem[BPSS14]{BPS}
	Fr{\'e}d{\'e}ric Bayart, Daniel Pellegrino, and Juan~B. Seoane-Sep{\'u}lveda.
	\newblock The {B}ohr radius of the n-dimensional polydisk is equivalent to
	{$\sqrt{(\log n)/n}$}.
	\newblock {\em Advances in Mathematics}, 264:726--746, 2014.
	

	\bibitem[Dan05]{Dankert}
	Christoph Dankert.
	\newblock Efficient Simulation of Random Quantum States and Operators.
	\newblock {\em arXiv preprint, arXiv:quant-ph/0512217,} 2005.
	
	\bibitem[DCEL09]{DCEL09}Christoph Dankert, Richard Cleve, Joseph Emerson, and Etera Livine.
\newblock Exact and approximate unitary 2-designs and their application to fidelity estimation.
\newblock {\em Phys. Rev. A} 80(1), 2009.

	
	
	\bibitem[DFKO07]{DFKO07}
	Irit Dinur, Ehud Friedgut, Guy Kindler, and Ryan O'Donnell.
	\newblock On the {F}ourier tails of bounded functions over the discrete cube.
	\newblock {\em Israel Journal of Mathematics}, 160(1):389--412, 2007.
	
	\bibitem[DFOC{\etalchar{+}}11]{DFOOS}
	Andreas Defant, Leonhard Frerick, Joaquim Ortega-Cerd{\`a}, Myriam
	Ouna{\"\i}es, and Kristian Seip.
	\newblock The {B}ohnenblust-{H}ille inequality for homogeneous polynomials is
	hypercontractive.
	\newblock {\em Annals of mathematics}, pages 485--497, 2011.
	
	\bibitem[DGMSP19]{DGMS}
	Andreas Defant, Domingo Garc{\'\i}a, Manuel Maestre, and Pablo Sevilla-Peris.
	\newblock {\em Dirichlet series and holomorphic functions in high dimensions},
	volume~37.
	\newblock Cambridge University Press, 2019.
	
	\bibitem[DMP19]{DMP}
	Andreas Defant, Mieczys{\l}aw Masty{\l}o, and Antonio P{\'e}rez.
	\newblock On the {F}ourier spectrum of functions on {B}oolean cubes.
	\newblock {\em Mathematische Annalen}, 374(1):653--680, 2019.
	
	\bibitem[EI22]{EI22}
	Alexandros Eskenazis and Paata Ivanisvili.
	\newblock Learning low-degree functions from a logarithmic number of random
	queries.
	\newblock In {\em Proceedings of the 54th Annual ACM SIGACT Symposium on Theory
		of Computing}, pages 203--207, 2022.
		
			
	\bibitem[EIS23]{EIS}
		Alexandros Eskenazis, Paata Ivanisvili, and Lauritz Streck.
		\newblock Low-degree learning and the metric entropy of polynomials.
		\newblock  {\em arXiv preprint, arXiv:2203.09659}, 2022
	
\bibitem[GAE07]{GAE}
D.~Gross, K.~Audenaert, and J.~Eisert.
\newblock Evenly distributed unitaries: On the structure of unitary designs.
\newblock {\em Journal of Mathematical Physics}, 48(5):052104, 2007.
  
  \bibitem[GZKC{\etalchar{+}}22]{PhysRevLett.129.160501}
Daniel Gonz\'alez-Cuadra, Torsten~V. Zache, Jose Carrasco, Barbara Kraus, and
  Peter Zoller.
\newblock Hardware efficient quantum simulation of non-abelian gauge theories
  with qudits on {R}ydberg platforms.
\newblock {\em Phys. Rev. Lett.}, 129:160501, Oct 2022.

	
\bibitem[HKP20]{Huang2020}
Hsin-Yuan Huang, Richard Kueng, and John Preskill.
\newblock Predicting many properties of a quantum system from very few
  measurements.
\newblock {\em Nature Physics}, 16(10):1050--1057, June 2020.
	
	\bibitem[HCP22]{HCP22}
	Hsin-Yuan Huang, Sitan Chen, and John Preskill.
	\newblock Learning to predict arbitrary quantum processes.
	\newblock {\em arXiv preprint arXiv:2210.14894}, 2022.
	
	
\bibitem[IRRR{\etalchar{+}}21]{fouriergrowth}
Siddharth Iyer, Anup Rao, Victor Reis, Thomas Rothvoss, and Amir Yehudayoff.
\newblock Tight bounds on the {F}ourier growth of bounded functions on the
  hypercube.
\newblock {\em Electron. Colloquium Comput. Complex.}, {TR21-102}, 2021.
  
\bibitem[KAJL{\etalchar{+}}21]{quant-sim}
Doga~Murat Kurkcuoglu, M.~Sohaib Alam, Joshua~Adam Job, Andy C.~Y. Li,
  Alexandru Macridin, Gabriel~N. Perdue, and Stephen Providence.
\newblock Quantum simulation of $\phi^4$ theories in qudit systems.
\newblock 2021.
  
  \bibitem[KSVZ24]{ksvzITCS}
 Ohad Klein, Joseph Slote, Alexander Volberg, and  Haonan Zhang.
\newblock Quantum and Classical Low-Degree Learning via a Dimension-Free Remez Inequality.
\newblock {\em 15th Innovations in Theoretical Computer Science Conference (ITCS 2024)}.
 

  
  \bibitem[LMN93]{LMN}
Nathan Linial, Yishay Mansour, and Noam Nisan.
\newblock Constant depth circuits, fourier transform, and learnability.
\newblock {\em J. ACM}, 40(3):607–620, jul 1993.
  
  \bibitem[O'Do14]{Odonnell}
Ryan O'Donnell.
\newblock {\em Analysis of Boolean Functions}.
\newblock Cambridge University Press, 2014.
  
  \bibitem[Ver18]{vershynin}
  Roman Vershynin.
  \newblock High-dimensional probability: An introduction with applications in data science.
  \newblock {\em Cambridge university press}, 47, 2018.
  
  
  \bibitem[WHSK20]{quditsurvey}
Yuchen Wang, Zixuan Hu, Barry~C. Sanders, and Sabre Kais.
\newblock Qudits and high-dimensional quantum computing.
\newblock {\em Frontiers in Physics}, 8, 2020.

			\bibitem[RWZ24]{RWZ24}Cambyse Rouzé, Melchior Wirth, Haonan Zhang.
		\newblock Quantum Talagrand, KKL and Friedgut's theorems and the learnability of quantum Boolean functions.
		\newblock {\em Commun. Math. Phys.} 405, 95 (2024).
		
		\bibitem[S06]{Scott_2006}Andrew J. Scott.
		\newblock Tight informationally complete quantum measurements.
		\newblock {\em J. Phys. A} 39, 45 (2006).

			
		\bibitem[SVZ23a]{SVZ} Joseph Slote, Alexander Volberg, and Haonan Zhang. 
		\newblock Noncommutative Bohnenblust--Hille inequality in the Heisenberg--Weyl and Gell-Mann bases with applications to fast learning. 
		\newblock arXiv: 2301.01438, v2.
		
		\bibitem[SVZ23b]{SVZgmp} Joseph Slote, Alexander Volberg, and Haonan Zhang. 
		\newblock A dimension free discrete Remez inequality on multi-tori.
		\newblock arXiv:2305.10828.
		
		
		\bibitem[VZ23]{VZ23} Alexander Volberg and Haonan Zhang.
		\newblock Noncommutative Bohnenblust--Hille inequality.
		\newblock {\em Math. Ann.}, 1-20, 2023.
\end{thebibliography}
\end{document}